\documentclass[11pt]{amsart}

\usepackage{amsmath, amsfonts, amssymb,color, appendix, epsfig,enumerate}

\newcommand{\mB}{\mathcal{B}}

\newcommand{\mF}{\mathcal{F}}

\newcommand{\mH}{\mathcal{H}}

\newcommand{\mJ}{\mathcal{J}}
\newcommand{\mK}{\mathcal{K}}
\newcommand{\mL}{\mathcal{L}}
\newcommand{\mM}{\mathcal{M}}

\newcommand{\mO}{\mathcal{O}}

\newcommand{\mS}{\mathcal{S}}

\newcommand{\mU}{\mathcal{U}}
\newcommand{\mV}{\mathcal{V}}
\newcommand{\mW}{\mathcal{W}}

\newcommand{\fa}{\mathfrak{a}}

\newcommand{\fb}{\mathfrak{b}}

\newcommand{\fd}{\mathfrak{d}}

\newcommand{\ff}{\mathfrak{f}}

\newcommand{\fM}{\mathfrak{M}}

\newcommand{\fN}{\mathfrak{N}}

\newcommand{\fp}{\mathfrak{p}}

\newcommand{\fq}{\mathfrak{q}}

\newcommand{\ft}{\mathfrak{t}}

\newcommand{\bfA}{\mathbf{A}}

\newcommand{\bfC}{\mathbf{C}}

\newcommand{\bfF}{\mathbf{F}}
\newcommand{\bfG}{\mathbf{G}}
\newcommand{\bfH}{\mathbf{H}}

\newcommand{\bfQ}{\mathbf{Q}}
\newcommand{\bfR}{\mathbf{R}}

\newcommand{\bfT}{\mathbf{T}}

\newcommand{\bfZ}{\mathbf{Z}}

\newcommand{\bfa}{\mathbf{a}}
\newcommand{\bfb}{\mathbf{b}}
\newcommand{\bfc}{\mathbf{c}}

\newcommand{\bff}{\mathbf{f}}

\newcommand{\bfk}{\mathbf{k}}

\newcommand{\Oo}{\mathcal{O}}

\newcommand{\OF}{\mathcal{O}_F}

\newcommand{\OK}{\mathcal{O}_K}
\newcommand{\OL}{\mathcal{O}_L}

\newcommand{\AF}{\mathbf{A}_F}
\newcommand{\AK}{\mathbf{A}_K}

\newcommand{\AQ}{\mathbf{A}}

\newcommand{\OFv}{\mathcal{O}_{F,v}}

\newcommand{\ov}{\overline}

\newcommand{\be}{\begin{equation}}
\newcommand{\ee}{\end{equation}}
\newcommand{\bes}{\begin{equation*}}
\newcommand{\ees}{\end{equation*}}

\newcommand{\bs}{\begin{split}}
\newcommand{\es}{\end{split}}
\newcommand{\bss}{\begin{split*}}
\newcommand{\ess}{\end{split*}}

\newcommand{\bmat}{\left[ \begin{matrix}}
\newcommand{\emat}{\end{matrix} \right]}
\newcommand{\bsmat}{\left[ \begin{smallmatrix}}
\newcommand{\esmat}{\end{smallmatrix} \right]}

\newcommand{\bml}{\begin{multline}}
\newcommand{\eml}{\end{multline}}
\newcommand{\bmls}{\begin{multline*}}
\newcommand{\emls}{\end{multline*}}

\DeclareMathOperator{\ad}{ad}
\DeclareMathOperator{\alg}{alg}

\DeclareMathOperator{\BC}{BC}

\DeclareMathOperator{\Cl}{Cl}

\DeclareMathOperator{\cond}{cond}

\DeclareMathOperator{\Frob}{Frob}
\DeclareMathOperator{\Gal}{Gal}
\DeclareMathOperator{\GL}{GL}
\DeclareMathOperator{\GSp}{GSp}

\DeclareMathOperator{\Mat}{Mat}

\DeclareMathOperator{\Nm}{N}

\DeclareMathOperator{\Res}{Res}

\DeclareMathOperator{\SL}{SL}
\DeclareMathOperator{\Sp}{Sp}

\DeclareMathOperator{\st}{st}
\DeclareMathOperator{\Sym}{Sym}

\DeclareMathOperator{\Tr}{Tr}
\DeclareMathOperator{\U}{U}
\DeclareMathOperator{\val}{val}
\DeclareMathOperator{\vol}{vol}

\newcommand{\tr}{\textup{tr}\hspace{2pt}}

\usepackage[all]{xy}
\usepackage{amsthm}
\usepackage{amssymb, amsfonts}
\SelectTips{cm}{10}\UseTips
\theoremstyle{plain}
\newtheorem{thm}{Theorem}
\newtheorem{prop}[thm]{Proposition}
\newtheorem{cor}[thm]{Corollary}
\newtheorem{lemma}[thm]{Lemma}
\newtheorem{conj}[thm]{Conjecture}

\theoremstyle{definition}
\newtheorem{definition}[thm]{Definition}

\newtheorem{rem}[thm]{Remark}

\numberwithin{thm}{section}
\numberwithin{equation}{section}

\newcommand{\up}{\upsilon}
\makeatletter
\let\@wraptoccontribs\wraptoccontribs
\makeatother

\begin{document}

\title[Congruence primes for automorphic forms]{Congruence primes for automorphic forms on unitary groups and applications to the arithmetic of Ikeda lifts}
\author[Jim Brown]{Jim Brown$^1$}
\address{$^1$Department of Mathematical Sciences\\
Clemson University\\ Clemson, SC 29634 USA}
\email{jimlb@g.clemson.edu}
\author[Krzysztof Klosin]{Krzysztof Klosin$^2$}
\address{$^2$Department of Mathematics\\
Princeton University\\
Fine Hall, Washington Road\\
Princeton NJ 08544-0001 USA}
\email{kklosin@qc.cuny.edu}

\footnotetext{The work of the second author was supported by a grant from the Simons Foundation (\#354890, Krzysztof Klosin) and Young Investigator Grant \#H98230-16-1-0129 from the National Security Agency. In addition the second author was partially supported by a PSC-CUNY Award, jointly funded by The Professional
Staff Congress and The City University of New York.}

\maketitle

\begin{abstract}
In this paper we provide a sufficient condition for a prime to be a congruence prime for an automorphic form $f$ on the unitary group $\U(n,n)(\AF)$ for a large class of totally real fields $F$ via a divisibility of a special value of the standard $L$-function associated to $f$. We also study $\ell$-adic properties of the Fourier coefficients of an Ikeda lift $I_{\phi}$ (of an elliptic modular form $\phi$) on $\U(n,n)(\AQ_{\bfQ})$ proving that they are $\ell$-adic integers which do not all vanish modulo $\ell$. Finally we combine these results to show that the condition of $\ell$ being a congruence prime for $I_{\phi}$ is controlled by the $\ell$-divisibility of a product of special values of the symmetric square $L$-function of $\phi$. We close the paper by computing an example when our main theorem applies.
\end{abstract}

\section{Introduction}

The problem of classifying congruences between automorphic forms has attracted a considerable amount of attention due not only to its inherent interest, but also because of the arithmetic applications of many of these congruences.  For instance, a congruence between an elliptic cusp form and an Eisenstein series implies that the residual Galois representation associated to the cusp form is  reducible and non-split.  This in turn implies one can construct certain field extensions with prescribed ramification properties. One can see Ribet's seminal paper on the converse to Herbrand's theorem for such a construction \cite{RibetInvent76}.  The classification of congruence primes for elliptic modular forms was given by Hida in a series of papers \cite{HidaInvent63-81, HidaInvent81,HidaAJM88} where he showed special values of the symmetric square $L$-function control such congruences.

As in the case of elliptic modular forms, congruences between automorphic forms on other reductive algebraic groups have numerous arithmetic applications. One can see \cite{AgarwalBrownMathZ14, AgarwalKlosinJNT13, BalasubramanyamRaghuram15, BergerCompMath09,BochererDummiganSchulzePillotJMSJ12, BrownCompMath07, BrownIMRN11,  DummiganIbukiyamaKatsurada,KatsuradaMathZ08, Katsuradacongruence,  KlosinAnnInstFourier2009, Klosin15,SkinnerUrbanJussieu06,SkinnerUrbanInventMath14}
for examples of such applications. As such, there is considerable interest in classifying congruence primes for automorphic forms on these groups. One would expect that the congruence primes should be controlled by certain special values of $L$-functions attached to the automorphic form. While this remains an open problem, partial progress has been made for automorphic forms of full level defined on symplectic groups in the form of a sufficient condition for a prime to be a congruence prime \cite{KatsuradaMathZ08}.  In this paper we provide a sufficient condition for an automorphic form $f$  (which is an eigenform for the Hecke algebra) defined on the unitary  group $\U(n,n)(\AF)$ to be congruent to an automorphic form on the same group that is orthogonal to $f$. This condition is expressed  in terms of the divisibility of a special value of the standard $L$-function associated to $f$. Here $n>1$ and  $F$ is any totally real finite extension of $\bfQ$ which has the property that every ideal capitulates in the imaginary quadratic extension over which $\U(n,n)_F$ splits (this is true for example if $F$ has class number one).  On the other hand most of the results referenced above \cite{AgarwalBrownMathZ14, AgarwalKlosinJNT13, BergerCompMath09,BochererDummiganSchulzePillotJMSJ12, BrownCompMath07, BrownIMRN11,  DummiganIbukiyamaKatsurada,KatsuradaMathZ08, Katsuradacongruence,  KlosinAnnInstFourier2009, Klosin15,SkinnerUrbanJussieu06,SkinnerUrbanInventMath14}
on congruence primes for automorphic forms on reductive groups require the automorphic forms to be defined over $\bfQ$.
The results of \cite{KatsuradaMathZ08} on congruence primes for automorphic forms defined on $\GSp(2n)(\bfA_{\bfQ})$
 also require the forms to have full level while our result holds for congruence subgroups of the form $\mK_{0,n}(\fN)$ for $\fN$ an ideal in $\mO_{F}$.


More precisely, we prove the following result. Let $\ell$ be a rational prime and fix embeddings $\ov{\bfQ} \hookrightarrow \ov{\bfQ}_{\ell} \hookrightarrow \bfC$. Let $K/F$ be an imaginary quadratic extension.
Let $\fM \subset \mO_{F}$ be an ideal and $f$ a cuspidal Hecke eigenform of level $\fM$ and parallel weight $k$ defined on  $\U(n,n)(\bfA_{F})$ (for precise definitions of level and weight see section \ref{Notation}).  Let $\xi$ be a Hecke character of $K$ of infinity type
$\xi_{\bfa}(z) = \left(\frac{z}{|z|}\right)^{-t}$ for $t = (t, \dots, t) \in \bfZ^{\bfa}$.
  We denote the standard $L$-function associated to $f$ twisted by $\xi$ by $L(s, f, \xi; \st)$.  We consider the value \begin{equation*}
L^{\rm alg}(2n+t/2, f, \xi;\st) = \frac{L(2n+t/2, f, \xi; \st)}{\pi^{nd(k + 2n + t +1)} \langle f, f \rangle}
\end{equation*}
where $d= [F:\bfQ]$, which was shown to be algebraic by Shimura \cite{ShimuraArithmeticity}.  The main result of the first part of this paper is that for a prime $\ell$ up to some technical conditions one has $f$ is congruent modulo $\ell^{b}$ to an automorphic form $f'$, orthogonal to $f$, if $-b = \val_{\ell}\left(\frac{\pi^{dn^2}}{\vol(\mF_{\mK_{0,n}(\fM)})}\ov{L^{\rm alg}(2n+t/2, f, \xi; \st)}\right) < 0$ where $\mF_{\mK_{0,n}(\fM)}$ is a fundamental domain for the congruence subgroup of level $\fM$ (please, see Remark \ref{volume explicit} where its volume is made explicit).  One can see Theorem \ref{thmmain} for a self-contained precise statement of the result.

In the second part of the paper we study arithmetic properties of  hermitian Ikeda lifts 
and then apply the results of the first part to construct a congruence between this lift and an orthogonal automorphic form on $\U(n,n)(\AQ)$ of full level. More specifically, let $K=\bfQ(\sqrt{-D_K})$ be an imaginary quadratic field of discriminant $-D_K$. Let $n>1$ be an integer  and $\phi$ be either a modular form of level $D_K$ and Nebentypus being the quadratic character  associated with $K$ (if $n$ is even) or a modular form of level one (if $n$ is odd). The Ikeda lift is an automorphic form $I_{\phi}$ on $\U(n,n)(\AQ_{\bfQ})$ of full level. It was proven in \cite{Ikeda08} that $I_{\phi}$ is non-zero unless $n \equiv 2 \pmod{4}$ and $\phi$ arises from a Hecke character of an imaginary quadratic field. In this paper we are interested in the non-vanishing of $I_{\phi}$ modulo a prime $\ell$  which is important for constructing congruences. We show that $I_{\phi}$ has Fourier coefficients which are algebraic integers (Proposition \ref{integers1}) and then prove these coefficients do not all vanish modulo $\ell$ unless $n\equiv 2$ (mod 4) and the residual (mod $\ell$) Galois representation attached to $\phi$ restricted to $G_K$  has abelian image (Theorem \ref{not cong to zero}). The condition on the image of the Galois representation is not surprising in light of a result of Ribet \cite{RibetModFuncOneVar77} which says that  modular forms whose $\ell$-adic (i.e., characteristic zero) Galois representations restricted to the absolute Galois group of an imaginary quadratic field have abelian image arise from Hecke characters. 
See section \ref{Arithm} for more details.

Given these arithmetic results and an inner product formula due to Katsurada \cite{Katsurada17}  we construct the aforementioned congruence in section \ref{Congruence to Ikeda lift}. More precisely, we show that up to some technical hypotheses $I_{\phi}$ is congruent to an orthogonal automorphic form $f'$ on $\U(n,n)(\AQ_{\bfQ})$ modulo $\ell^b$ where
\be  b=\val_{\ell}(\mV \eta_{\phi}) \quad \textup{with}\quad
\mV:= \begin{cases} \prod_{i=2}^n\frac{L(i+2k-1, \Sym^2 \phi \otimes \chi_K^{i+1})}{\pi^{2k+2i-1}\Omega_{\phi}^+\Omega_{\phi}^-} & n=2m+1\\ \prod_{i=2}^n \frac{L(i+2k, \Sym^2 \phi \otimes \chi_K^{i})}{\pi^{2k+2i}\Omega_{\phi}^+\Omega_{\phi}^-} & n=2m \end{cases}
\ee
and $\eta_{\phi}$ is a generator of the congruence module of $\phi$.
 Here we require the choice of an auxiliary character $\xi$ so that $\ell$ does not divide a  product of certain special values of an $L$-function twisted by $\xi$. (See Theorem \ref{Ikedacong} for a precise statement.)  While this result is in the spirit of the results of Katsurada on congruence primes for Ikeda lifts on $\GSp(2n)(\bfA_{\bfQ})$ \cite{Katsuradacongruence},
 our result provides more freedom in ``missing'' the $L$-values as we can twist a large family of Hecke characters of $K$ while Katsurada's result allows one only to vary the evaluation point of the $L$-functions.
We then prove that replacing $\val_{\ell}(\mV\eta_{\phi})$ by $\val_{\ell}(\mV)$ 
we can ensure that $f'$ is not an Ikeda lift itself. This result however requires some expected (but so far to the best of our knowledge absent from the literature) properties of the Hecke algebra and the Galois representations attached to automorphic forms on $\U(n,n)(\AF)$ to hold. For details see section \ref{congint}.  One can see section \ref{sec:examples} for explicit examples of the main congruence result on Ikeda lifts. Our methods should apply directly to $\GSp(2n)(\bfA_{F})$ as well where $F$ is a totally real extension of $\bfQ$; this will be the subject of future work.

We include a discussion on how the constructed congruence can be applied to provide evidence for a new case of the Bloch-Kato conjecture, which concerns the motive $\ad^0\rho_{\phi}(3)$, where $\rho_{\phi}$ is the Galois representation attached to $\phi$. Such evidence for the motive $\ad^0\rho_{\phi}(2)$ was provided in \cite{KlosinAnnInstFourier2009, Klosin15}. We refrain ourselves however from formulating a precise theorem because we can only prove it under an assumption that certain cohomology classes lie in the correct eigenspace of the complex conjugation and such a result appears to us less than satisfying on the one hand and clumsy to state on the other (cf. Section \ref{Speculations about consequences for the Bloch-Kato conjecture}).

We would like to thank Thanasis Bouganis for sending us a copy of his paper \cite{Bouganis15preprint} before it was published and David Loeffler for a helpful discussion of the example found in \cite[Section 6]{LoefflerZerbesIwasawa}.  We are also indebted to the anonymous referee for a very careful reading of the paper and numerous suggestions which greatly improved the exposition in the article and prompted us to include section 9.

\section{Notation and terminology} \label{Notation}

\subsection{Number fields and Hecke characters} We fix once and for all an algebraic closure $\ov{\bfQ}$ and an embedding $\iota_{\infty}: \ov{\bfQ} \hookrightarrow \bfC$.
For $z \in \bfC$ we write $e(z) = e^{2 \pi i z}$.
Given a number field $L$, we write $\mO_{L}$ for the ring of integers of $L$.  Given a prime $w$ of $L$, we write $L_{w}$ for the completion of $L$ at $w$, $\mO_{L,w}$ for the valuation ring of $L_{w}$, and $\varpi_{w}$ for a uniformizer. If $L/E$ is an extension of number fields and $\up$ a prime of $E$, we write $L_{\up} = E_{\up} \otimes_{E} L$ and $\mO_{L,\up}  = \mO_{E,\up} \otimes_{\mO_{E}} \mO_{L}$.  We also set $\widehat{\mO}_{L} = \prod_{w \nmid \infty} \mO_{L,w}$. Given a finite prime $w$ of $L$, we let $\val_{w}$ denote the $w$-adic valuation on $L_{w}$. For $\alpha \in L_{w}$, we set $|\alpha|_{L_{w}} = q^{-\val_{w}(\alpha)}$ where $q$ is the cardinality of the residue class field at $w$. We will write $|\alpha|_{w}$ for $|\alpha|_{L_{w}}$ if it is clear from context what is meant.

Let $\bfA_{L}$ denote the adeles of $L$, $\bff_{L}$ the set of finite places of $L$, and $\bfa_{L}$ the set of embeddings of $L$ into $\ov{\bfQ}$. We regard $\bfa_{L}$ as the set of infinite places of $L$ via the
fixed embedding $\iota_{\infty}: \ov{\bfQ} \hookrightarrow \bfC$. We write $\bfA_{L,\bfa_{L}}$ and $\bfA_{L,\bff_{L}}$ for the infinite and finite part of $\bfA_{L}$ respectively.  For $\alpha = (\alpha_w) \in \bfA_{L}$, set $|\alpha|_{L} = \prod_{w} |\alpha_{w}|_{L_{w}}$.  For $\alpha \in \bfA_{L}^{\times}$ or $\alpha \in \bfC^{\bfa_{L}}$ and $c \in \bfR^{\bfa_{L}}$, we set $\alpha^{c} = \prod_{w \in \bfa_{L}} \alpha_{w}^{c_{w}}$. 


We say $\psi$ is a Hecke character of $L$ if $\psi$ is a continuous homomorphism from $\bfA_{L}^{\times}$ to $\{z \in \bfC : |z| = 1\}$ so that $\psi(L^{\times}) = 1$. We write $\psi_{\up}, \psi_{\bff_{L}}$ and $\psi_{\bfa_{L}}$ for its restrictions to $L_{\up}^{\times}$, $\bfA_{L,\bff_{L}}^{\times}$, and $\bfA_{L,\bfa_{L}}^{\times}$ respectively. Given $\psi$, there is a unique ideal $\ff \subset \mO_{L}$ satisfying $\psi_{\up}(\alpha) =1$ if $\up \in \bff_{L}$, $\alpha \in \mO_{L,\up}^{\times}$ and $\alpha - 1 \in \ff_{\up}$ and if $\ff'$ is another ideal of $\mO_{L}$ with this property then $\ff' \subset \ff$. The ideal $\ff$ is referred to as the conductor of $\psi$.  Given an integral ideal $\fN$ of $\OL$ and a Hecke character $\psi$, we write
\begin{equation*}
\psi_{\fN} = \prod_{\substack{\up \mid \fN}} \psi_{\up}.
\end{equation*}

Let $F$ be a totally real extension of $\bfQ$ and $K$ an imaginary quadratic extension of $F$.   Throughout the paper we let $\bfa$ (resp. $\bfb$) denote $\bfa_{F}$ (resp. $\bfa_{K}$) and $\bff$ (resp. $\bfk$) denote $\bff_{F}$ (resp. $\bff_{K}$).  Let $D_{K}$ denote the discriminant of $K$.
Set $\chi_{K}$ to denote the quadratic character associated with the extension $K/F$.  We identify $F_{\bfa} := \bfR \otimes_{\bfQ} F$ with $\bfR^{\bf a}:= \prod_{\up \in {\bf a}} \bfR$ via the map $a \mapsto (\sigma(a))_{\sigma \in {\bf a}}$. 

\subsection{Unitary group}
Let $\bfG_{a}$ denote the additive group scheme and $\bfG_{m}$ the multiplicative group scheme. Let $\Mat_n$ denote the $\bfZ$-group scheme of $n\times n$ matrices ($\cong \bfG_a^{n^2}$). Given a matrix $A \in \Res_{\OK/\OF}\Mat_{n/\OK}$, we set $A^{*} = \ov{A}^{t}$ and $\hat{A} = (A^{*})^{-1}$ where bar denotes the action of the nontrivial element of $\Gal(K/F)$. Associated to the imaginary quadratic extension $K/F$ we have the unitary similitude group scheme over $\OF$:
\begin{equation*}\textrm{GU}(n,n) = \{ A \in \Res_{\mO_{K}/\mO_{F}} \GL_{2n/\OK} : A J_{n} A^{*} = \mu_{n}(A) J_{n}\}
\end{equation*}
where $J_{n} = \bmat & -1_{n} \\ 1_{n} & \emat$, $1_{n}$ is the $n \times n$ identity matrix, and $\mu_n: \Res_{\mO_{K}/\mO_{F}} \GL_{2n/\mO_{K}} \rightarrow \bfG_{m/\mO_{F}}$ is a morphism of $\mO_{F}$-group schemes. Set $G_{n} = \ker \mu_{n}$.
For a $2n\times 2n$ matrix $g$ we will write $a_g, b_g, c_g, d_g$ to be the $n\times n$ matrices defined by $g=\bmat a_g & b_g \\ c_g & d_g \emat.$

Write $P_{n}$ for the standard Siegel parabolic of $G_{n}$, i.e.,
\begin{equation*}
P_{n} = \left\{ g = \bmat a_{g} & b_{g} \\ 0_n & d_{g} \emat \in G_{n}\right\}.
\end{equation*}
We have the Levi decomposition of $P_{n} = M_{P_{n}} N_{P_{n}}$ where
\begin{equation*}
M_{P_{n}} = \left\{ \bmat A & \\ &  \hat{A} \emat : A \in \Res_{\OK/\OF}\GL_{n/\OK} 
\right\}
\end{equation*}
and
\begin{equation*}
N_{P_{n}} = \left\{ \bmat 1_{n} & S \\   & 1_{n} \emat: S \in \Res_{\OK/\OF}\Mat_{n/\OK}, S^{*}= S \right\}.
\end{equation*}

Let $\fN \subset \mO_{F}$ be an ideal.   For $v \in \bff$ 
 we set
\begin{equation*}
\mK_{0,n,\up}(\fN) = \{ g \in G_{n}(F_{\up}): a_{g}, b_{g}, d_{g} \in \Mat_{n}(\mO_{K,\up}), c_{g} \in \Mat_{n}(\fN \mO_{K,\up})\}
\end{equation*}
and
\begin{equation*}
\mK_{1,n,\up}(\fN) = \{ g \in \mK_{0,n,\up}(\fN): a_{g} - 1_{n} \in \Mat_{n}(\fN \mO_{K,\up})\}.
\end{equation*}
We put $\mK_{0,n,{\bf f}}(\fN) = \prod_{\up \in \bff} \mK_{0,n,\up}(\fN)$ and similarly for $\mK_{1,n,{\bf f}}(\fN)$. Set
\begin{equation*}
\mK_{0,n,\bfa}^{+} = \left\{ \bmat A & B \\ - B & A\emat \in G_{n}(\bfR^{\bf a}) : A, B \in \GL_{n}(\bfC^{\bf a}), A A^{*} + BB^{*} = 1_{n}, AB^{*} = B A^{*} \right\}
\end{equation*}
and $\mK_{0,n,\bfa}$
 to be the subgroup of $G_{n}(\bfR^{\bfa})$ generated by $\mK_{0,n,\bfa}^{+}$ and $J_{n}$. Set
$\mK_{0,n}(\fN) = \mK_{0,n,\bfa}^{+} \mK_{0,n,\bff}(\fN)$ and $\mK_{1,n}(\fN) = \mK_{0,n,\bfa}^{+} \mK_{1,n,\bff}(\fN)$.

%
%
%
%
%

\subsection{Automorphic forms}
Define
\begin{equation*}
\bfH_{n}= \{Z \in \Mat_{n}(\bfC): - i1_{n}(Z - Z^{*}) > 0 \}.
\end{equation*}
For $g_{\bfa} = (g_{v})_{v \in \bfa} \in G_{n}(\bfR^{\bfa})$ and $Z = (Z_{v})_{v \in \bfa} \in \bfH_{n}^{\bfa}$,
set $j(g_{\bfa}, Z) = (j(g_{v}, Z_{v}))_{v \in \bfa}$ where $j(g_{v}, Z_{v}) = \det(c_{g_{v}}Z_{v} + d_{v})$.

Let $\mK$ be an open compact subgroup of $G_{n}(\bfA_{F, \bf f})$. For $k, \nu \in \bfZ^{\bf a}$ let $\mM_{n,k,\nu}(\mK)$ denote the $\bfC$-space of functions $f: G_{n}(\bfA_{F}) \rightarrow \bfC$ satisfying the following:
\begin{itemize}
\item[(i)] $f(\gamma g) = f(g)$ for all $\gamma \in G_{n}(F), g \in G_{n}(\bfA_{F})$,
\item[(ii)] $f(g \kappa) = f(g)$ for all $\kappa \in \mK, g \in G_{n}(\bfA_{F})$,
\item[(iii)] $f(gu) = (\det u)^{-\nu} j(u, i 1_{n})^{-k} f(g)$ for all $g \in G_{n}(\bfA_{F}), u \in \mK_{0,n,\bfa}$,
\item[(iv)] $f_c(Z) = (\det g_{\bfa})^{\nu} j(g_{\bfa}, i 1_{n})^{k} f(g_{\bfa}c)$ is a
holomorphic function of $Z = g_{\bfa} i 1_{n} \in \bfH_{n}^{\bfa}$ for every $c \in G_{n}(\bfA_{F, {\bf f}})$ where $g_{\bfa}\in G_{n}(\bfR^{\bfa})$.
\end{itemize}

If $n=1$, one also needs to require that the $f_c$ in (iv) be holomorphic at cusps (cf. \cite[p. 31]{ShimuraArithmeticity} for what this means). We let $\mS_{n,k,\nu}(\mK)$ denote the space of cusp forms in $\mM_{n,k,\nu}(\mK)$ (cf. \cite[p. 33]{ShimuraArithmeticity}).
If $\nu = 0$ we drop it from the notation.

Let $\psi$ be a Hecke character of $K$ of conductor dividing $\fN$.  For $k, \nu \in \bfZ^{\bfa}$ we set
\begin{align*}
\mM_{n,k,\nu}(\fN, \psi):=\{&f \in \mM_{n, k, \nu}(\mK_{1, n, {\bf f}}(\fN)) \mid f(g\kappa) = \psi_{\fN}(\det (a_{\kappa}))^{-1} f(g), \\ &g\in G_n(\AF), \kappa \in \mK_{0, n,  {\bf f}}(\fN)\}.
\end{align*}
If $\psi$ is trivial we drop it from the notation.

Every automorphic form $f \in \mM_{n,k,\nu}(\mK)$ has a Fourier expansion which we now discuss. We first define $e_{\bfA_{F}}$ as follows.  Let $\alpha = (\alpha_{\up}) \in \bfA_{F}$, where $\up$ runs over places of $F$.  We set $e_{\up}(\alpha_{\up}) = e(\alpha_{\up})$ for $\up \in \bfa$ and write $e_{\bfa}(\alpha) = e\left(\sum_{\up \in \bfa} \alpha_{\up}\right)$.  If $\up$ is finite, we set $e_{\up}(\alpha_{\up}) = e(-y)$ where $y \in \bfQ$ is chosen such that $\Tr_{F_{\up}/\bfQ_{p}}(\alpha_{\up}) - y \in \bfZ_{p}$ if $\up \mid p$.   We then set $e_{\bfA_{F}}(\alpha) = \prod_{\up} e_{\up}(\alpha_{\up})$. For every $q \in \GL_{n}(\bfA_{K, \bfk})$ and $h \in S_{n}(F)$ there exist complex numbers $c_{f}(h,q)$ such that one has (a Fourier expansion of $f$)
\begin{equation*}
f\left(\bmat 1_{n} & \sigma \\ & 1_{n} \emat \bmat q \\ & \hat{q} \emat \right) = \sum_{h \in S_{n}(F)} c_{f}(h,q) e_{\bfA_{F}}(\tr h \sigma)
\end{equation*}
for every $\sigma \in S_{n}(\bfA_{F})$ where
\begin{equation*}
S_{n} = \{h \in \Res_{\mO_{K}/\mO_{F}} \Mat_{n/\mO_{K}}: h^{*} = h\}.
\end{equation*}
One can see \cite[Chapter III, Section 18]{ShimuraCBMS97} for further discussion of such Fourier expansions.

Our next step is to define what we mean by a congruence between two automorphic forms.  One knows (see e.g., \cite[Theorem 3.3.1]{BumpCambridgeUniversityPress88}) that for any finite subset $\mB$ of $\GL_{n}(\bfA_{K,\bfk})$
 of cardinality $h_{K}$ with the property that the canonical projection $c_{K}: \bfA_{K}^{\times} \rightarrow {\rm Cl}_{K}$ restricted to $\det \mB$ is a bijection, the following decomposition holds
\begin{equation*}
\GL_{n}(\bfA_{K}) = \bigsqcup_{b \in \mB} \GL_{n}(K) \GL_{n}^{+}(K_{\bfb}) b \GL_{n}(\widehat{\mO}_{K}).
\end{equation*}
Here $+$ refers to a totally positive determinant.
Such a set $\mB$ will be called a {\it base}.

Let $\Cl_K^-$ denote the subgroup of $\Cl_K$ on which the non-trivial element of $\Gal(K/F)$ acts by inversion.

\begin{lemma}\label{corl:nicebase} If  $(h_{K},2n) = 1$ and $\Cl_K^-=\Cl_K$ then  a base $\mB$ can be chosen so that all $b \in \mB$ (and hence also $p_{b} := \bmat b & \\ & \hat{b} \emat$)  are scalar matrices and $b b^{*} = b^{*} b = 1_{n}$ (and hence also $p_bp_b^* = p_b^*p_b=1_{2n}$).
\end{lemma}

\begin{proof} This is just a modification of the proof of \cite[Corollary 3.9]{Klosin15}. Identifying $\Cl_K$ with a quotient of $\AK^{\times}/K^{\times}$ and using the fact that raising to the power $n$ acts as automorphism on $\Cl_K$  since $(h_{K},n)=1$, it is enough to show that every element of $\Cl_K$ can be written in the form $\alpha\ov{\alpha}^{-1}$ for some $\alpha \in \AK^{\times}$. Let $\alpha_1, \dots, \alpha_{h_K}$ be representatives of $\Cl_K$. Since $2\nmid h_K$, $\alpha_1^2, \dots, \alpha_{h_K}^2$ also form a complete set of representatives. By our assumption that $\Cl_K=\Cl_K^-$ we get that for each $1 \leq i \leq h_K$ we have $\ov{\alpha}_i^{-1} = \alpha_i k_i$ for some $k_i \in K^{\times}$.   \end{proof}

\begin{rem} The assumption $\Cl_K=\Cl_K^-$ is satisfied if the image of the canonical map  $\Cl_K \to \Cl_F$ given by $\fa \mapsto \fa \ov{\fa}$ is trivial, so   in particular when $h_F=1$. 
Indeed, then $\ov{\fa} = \fa^{-1} k$ for some $k \in F$. The assumption $\Cl_K^-=\Cl_K$ is therefore weaker for allowing $k \in K$ and in fact since $h_K$ is odd it  is equivalent to the assumption that every ideal of $\OF$ capitulates in $\OK$. \end{rem}

%
%


\begin{definition}  Let $\mB$ be a base.  We say $\mB$ is {\it admissible} if
 all $b \in \mB$ are scalar matrices with $b b^{*} = 1_{n}$ (and hence also $p_{b}$ are scalar matrices with $p_{b} p_{b}^{*} = 1_{2n}$). Furthermore given a prime $\ell$ we say that $\mB$ is \emph{$\ell$-admissible}, if it is admissible and
 for every $b \in \mB$ there exists a rational prime $p \nmid 2 D_{K} \ell$ such that $b_{\fq} = 1_{n}$ for all $\fq \nmid p$ and $b_{\bfb} = 1_{n}$.
The set of primes $\fp$ for which $b_{\fq} \neq 1_{n}$ with $\fq \mid \fp$ will be called the support of $\mB$.
\end{definition}

Note that if $(h_K,2n)=1$ and $\Cl_K^-=\Cl_K$, then the Tchebotarev density theorem combined with Lemma \ref{corl:nicebase} implies that for every $\ell \nmid D_K$ an $\ell$-admissible base always exists.

From now on for the rest of the article for every rational prime $\ell$ we fix compatible embeddings $\ov{\bfQ} \hookrightarrow \ov{\bfQ}_{\ell} \hookrightarrow \bfC$.
\begin{definition} \label{def of cong} Let $\ell$ be a rational prime and $\mO$ the ring of integers in some algebraic extension $E$ of $\bfQ_{\ell}$. 

\begin{enumerate}
\item Let $f \in \mM_{n,k, \nu}(\fN,\psi)$.  We say $f$ has {\it $\mO$-integral} Fourier coefficients (with respect to $\mB$) if there exists a base $\mB$ so that for all $b \in \mB$ and all $h \in S_{n}(F)$ we have $e_{\bfa}(-i \tr h_{\bfa})c_{f}(h,b) \in \mO$.
\item Let $f, g \in \mM_{n,k, \nu}(\fN,\psi)$ and suppose that $f$ and $g$ both have $\Oo$-integral Fourier coefficients with respect to a base $\mB$. Let $E'$ be a finite extension of $\bfQ_{\ell}$ with $E' \subset E$ and write $\Oo'$ for its ring of integers and $\varpi$ for a uniformizer. We say $f$ is congruent to $g$ modulo $\varpi^{n}$ if for all $b \in \mB$ and all $h \in S_{n}(F)$ we have $e_{\bfa}(-i \tr h_{\bfa})c_{f}(h,b) -e_{\bfa}(-i \tr h_{\bfa})c_{g}(h,b) \in \varpi^n\Oo'$.  We denote this by $f \equiv g \pmod{\varpi^{n}}$.
\end{enumerate}
\end{definition}

\begin{rem} \label{not in number field} Our definition of a congruence between automorphic forms on $G_n(\AF)$ allows for the situation when one knows that each individual Fourier coefficient of $f$ is an algebraic integer, but one does not know if collectively they generate a number field. This is certainly known for modular forms on $\GL_2$ due to a straightforward connection between Hecke eigenvalues and Fourier coefficients. The connection in the case of higher rank groups is much more delicate. \end{rem}

\begin{rem} If $f$ and $g$ are congruent (mod $\varpi^n$) with respect to one
admissible base, say $\mB$, and $\mB'$ is another admissible base such that $f$ and $g$ both have $\Oo$-integral Fourier coefficients with respect to $\mB'$, then we have $e_{\bfa}(-i \tr h_{\bfa})c_{f}(h,b) \equiv e_{\bfa}(-i \tr h_{\bfa})c_{g}(h,b) \pmod{\varpi^{n}}$ for all $b\in \mB'$ and all $h \in S_n(F)$. This is easily proved using admissibility and a generalization (to the setting of totally real $F$ and rank $n$) of formula (5.4) in \cite{Klosin15} using the references indicated in the proof of Lemma 5.5 in \cite{Klosin15}. Formula (5.4) in [loc.cit.] can also be used to show that if $f$ and $g$ have $\Oo$-integral Fourier coefficients with respect to one admissible base, then they have $\Oo$-integral Fourier coefficients with respect to all admissible bases. Hence we see that the congruence relation between automorphic forms is transitive, and thus an equivalence relation. \end{rem}

\subsection{$L$-functions}

We collect here the definitions of the $L$-functions needed for this paper.
Let $F$ be a number field and, as before, let $\bff_{F}$ denote the set of all finite places of $F$.

Given an integral ideal $\fN$ of $\OF$ and an Euler product
\begin{equation*}
L(s) = \prod_{\up \in \bff_{F}} L_{\up}(s),
\end{equation*}
we write
\begin{equation*}
L^{\fN}(s) = \prod_{\substack{\up \in \bff_{F} \\ \up \nmid \fN}} L_{\up}(s)
\end{equation*}
and
\begin{equation*}
L_{\fN}(s) = \prod_{\substack{\up \mid \fN}} L_{\up}(s).
\end{equation*}

Let $\chi$ be a Hecke character of $F$.  We set
\begin{equation*}
L(s,\chi) = \prod_{\substack{\up \in \bff_{F}\\ v \nmid \cond (\chi)}} (1 - \chi(\varpi_{\up}) |\varpi_{\up}|_{\up}^{s})^{-1},
\end{equation*}
where we identify a uniformizer $\varpi_v$ with its image in $\AF^{\times}$.

Now, let $F$ be a totally real field, $K$ an imaginary quadratic extension of $F$ and $\psi$ and $\chi$ Hecke characters of $K$.
Let $f \in \mS_{n,k,\nu}(\fN, \psi)$ be an eigenform with parameters $\lambda_{\up,1}, \dots, \lambda_{\up,n}$ determined by the Satake isomorphism \cite[Section 20.6]{ShimuraArithmeticity}.  The standard $L$-function of $f$ is defined as in \cite[p. 171]{ShimuraArithmeticity} by
\begin{equation*}
L(s,f,\chi;\st) = \prod_{\substack{\up \in \bff\\ \up \nmid \mO_{F} \cap \cond(\chi)}} L_{\up}(s, f,\chi;\st)
\end{equation*}
where for $\up \in \bff$ and $\up \nmid \fN$ we set
\begin{align*}
L_{\up}&(s,f,\chi;\st)\\
 &= \left\{\begin{array}{ll}  \prod_{i=1}^{n} \left [ (1 - \lambda_{\up,i} \chi(\varpi_w)|\varpi_w|_{w}^{s-n+1})(1 - \lambda_{\up,i}^{-1}\chi(\varpi_w)|\varpi_w|_{w}^{s-n})\right ]^{-1} & K\otimes_F F_v \hspace{2pt} \textup{is a field} \\
    \prod_{i=1}^{2n} \left [ (1 - \lambda_{\up,i}^{-1} \chi(\varpi_{v_1})|\varpi_{v}|_{v}^{s-2n})(1 - \lambda_{\up,i}\chi(\varpi_{v_2})|\varpi_{v}|_{v}^{s+1})\right ]^{-1} & K\otimes_F F_v \cong F_v \times F_v  \end{array}\right.
\end{align*}
and when $\up \mid \fN$ we set
\begin{align*}
L_{\up}&(s,f,\chi;\st)\\
    &= \left\{ \begin{array}{ll} \prod_{i=1}^{n} \left[1 - \lambda_{\up,i} \chi(\varpi_w)|\varpi_{w}|_{w}^{s-n+1}\right]^{-1} & K\otimes_F F_v \hspace{2pt} \textup{is a field} \\
        \prod_{i=1}^{n} \left[(1 - \lambda_{\up,i} \chi(\varpi_{w_1})|\varpi_{w_1}|_{w_1}^{s-n+1})(1 - \lambda_{\up,n+i} \chi(\varpi_{w_2})|\varpi_{w_2}|_{w_2}^{s-n+1})\right]^{-1} &K\otimes_F F_v \cong F_v \times F_v. \end{array} \right.
\end{align*}
Here $\varpi_v$ denotes a uniformizer of $F_v$, $\varpi_w$ denotes a uniformizer of $K\otimes F_v$ if the latter is a field and $\varpi_{v_1}, \varpi_{v_2} \in K \otimes_F F_v$ correspond to $(\varpi_v,1)$ and $(1, \varpi_v)$ respectively under the isomorphism $K\otimes_F F_v \cong F_v \times F_v$ if such an isomorphism holds. Here again we identify $K \otimes_F F_v$ with its image in $\AK^{\times}$.

\section{An inner product relation}

Fix a totally real extension $F/\bfQ$ and an imaginary quadratic extension $K/F$.
Let $d = [F:\bfQ]$.  In this section we generalize a certain Rankin-Selberg type formula which for the case of $F=\bfQ$ was derived in \cite[Section 7]{Klosin15}. Since all the calculations carried out in [loc.cit.] carry over without difficulty to the general case,  we include here only the necessary setup as well as indicate where changes need to be made and  refer the reader to [loc.cit.] for a less condensed version (see also \cite{Bouganis15preprint}).

\subsection{Eisenstein series}

In this section we define the hermitian Siegel Eisenstein series and recall its properties needed to prove our congruence result.

Let $\fN \subset \mO_{F}$ be an ideal. Let $X_{m,\fN}$ be the set of Hecke characters $\psi'$ of $K$   of conductor dividing $\fN$ satisfying
\begin{equation}\label{eqn:7.4}
\psi'_{\bfa}(x) = (x/|x|)^{m}
\end{equation}
for $m \in \bfZ^{{\bf a}}$.  Here and later we regard Hecke characters of K as characters of $(\textup{Res}_{\OK/\OF}\mathbf{G}_{m/\OK})(\AF)$.
Let $\psi \in X_{m, \fN}$.
Recall we set $\psi_{\fN} = \prod_{\up \mid \fN} \psi_{\up}$.

For a place $\up$ of $F$ and an element $p = \bmat a_{p} & b_{p} \\ 0 & d_{p} \emat \in P_n(F_{\up})$, set
\begin{equation*}
\delta_{P_{n},\up}(p) = |\det (d_{p} d^{*}_{p})|_{v}
\end{equation*}
and $\delta_{P_{n}} = \prod_{\up} \delta_{P_{n},\up}$.  Set $\mu_{P_{n}} = \prod_{\up} \mu_{P_{n},\up}: G_n(\AF) \to \bfC$
to be a function vanishing outside of $P_n(\AF)\mK_{0,n}(\fN)$  where for $p_{\up} \in P_n(F_v)$
and $k_v \in \mK_{0,n,v}(\fN)$ we set
\begin{equation*}
\mu_{P_{n},\up}(p_{\up} k_{\up}) = \left\{ \begin{array}{ll} \psi_{\up}(\det d_{q_{\up}})^{-1} &  v \in \bff, \up \nmid \fN \\
        \psi_{\up}(\det d_{q_{\up}})^{-1} \psi_{\up} (\det d_{k_{\up}}) & v \in \bff, \up \mid \fN \\
        \psi_{\bfa}(\det d_{q_{\bfa}})^{-1} j(k_{\bfa}, i)^{-m} & \up \in \bfa. \end{array}\right.
\end{equation*}

The hermitian Siegel Eisenstein series  of weight $m \in \bfZ^{\bf a}$, level $\fN$, and character $\psi$ is defined by
\begin{equation*}
E(g,s,m,\psi, \fN) = \sum_{\gamma \in P_{n}(F)\backslash G_{n}(F)} \mu_{P_{n}}(\gamma g) \delta_{P_{n}}(\gamma g)^{-s}
\end{equation*}
for $\textup{Re}(s) >>0$. The
meromorphic  continuation of $E(g,s,m,\psi, \fN)$ to all $s \in \bfC$ is given by Shimura \cite[Proposition 19.1]{ShimuraCBMS97}.
We will make use of the following normalized Eisenstein series
\begin{equation*}
D(g,s,m,\psi,\fN) = \prod_{j=1}^{n} L^{\fN}(2s-j+1,\psi_{F}\chi_{K}^{j-1}) E(g,s,m,\psi,\fN)
\end{equation*}
where $\psi_{F}$ denotes the restriction of $\psi$ to $\bfA_{F}^{\times}$.

Let $\eta_{{\bff}} \in G_{n}(\bfA_{F})$ be the matrix with the infinity components equal to $1_{2n}$ and finite components equal to $J_{n} = \bmat & -1_{n} \\1_{n} \emat$. Given a function $f: G_{n}(\bfA_{F}) \rightarrow \bfC$, set $f^{*}(g) = f(g \eta_{\bff}^{-1})$.  If $f \in \mM_{n,k}(\mK)$, we have $f^{*} \in \mM_{n,k}(\eta_{\bff}^{-1} \mK \eta_{\bff})$.

Write
\begin{equation*}
D^{*}\left(\bmat q & \sigma q \\ & \hat{q} \emat, n - m/2, m, \psi, \fN \right) = \sum_{h \in S_{n}(F)} c_{D^{*}}(h,q) e_{\bfA_{F}}(h \sigma).
\end{equation*}
We have the following result giving the integrality of the Fourier coefficients of $D^{*}$.

\begin{prop}\label{thm:eiscoeffs} Let $\ell$ be a prime and assume $\ell \nmid D_{K} \Nm_{F/\bfQ}(\fN)(n-2)!$, $(h_K, 2n)=1$, and that $\Cl_K^-=\Cl_K$. Let $\mB$ be an $\ell$-admissible base  whose support is relatively prime to $\cond(\psi)$.   There exists a finite extension $E$ of $\bfQ_{\ell}$ so that
\begin{equation*}
\pi^{-n(n+1)d/2} e_{\bfa}(-i \tr h_{\bfa}) c_{D^{*}}(h,b) \in \mO_{E}
\end{equation*}
for all $h \in S_{n}(F)$ and all $b \in \mB$.
\end{prop}

\begin{proof}  This is proved as in \cite[Theorem 7.8]{Klosin15} and \cite[Corollary 7.11]{Klosin15}.  There are only two significant changes to the proof given in \cite{Klosin15}. The first is
the power of $\pi$ has an additional $d=[F: \bfQ]$, which follows immediately from \cite[Equation 18.11.5]{ShimuraCBMS97}. The second is the reference that $L(n,\psi')$ lies in the ring of integers of a finite extension of $\bfQ_{\ell}$ for all $n \in \bfZ_{<0}$ must be changed from \cite[Corollary 5.13]{Washingtonbook} to \cite[Theorem 1, p. 104]{HidaLondonMathSoc93} as our characters are now defined over $F$ instead of $\bfQ$.
\end{proof}

\subsection{Theta series}\label{sec:theta}
Let $k \in \bfZ^{{\bf a}}$ be such that $k_{\up} > 0$ for all $\up \in {\bf a}$ and let $\xi$ be a Hecke character of $K$ with conductor $\ff_{\xi}$ and infinity type $|x|^{t} x^{-t}$ for $t \in \bfZ^{{\bf a}}$ with $t_{\up} \geq -k_{\up}$ for each $\up \in {\bf a}$. In \cite[Section A5]{ShimuraArithmeticity} and \cite[Section A7]{ShimuraCBMS97} Shimura defines a theta series on $G_n(\AQ)$ associated to a quadruple consisting of $k$ (as above), the Hecke character $\xi$, a matrix $r \in \GL_n(\bfA_{K,\bfk})$  and a matrix $\tau \in S^{+}_{n}(F) = \{h \in S_{n}(F)\subset S_n(\AF): h_{\up} > 0 \, \textrm{for all}\, \up \in \bfa\}$ which satisfies the following condition \be \label{reldiff} \{g^* \tau g \mid g \in r\OK^n\} =\fd_{F/\bfQ}^{-1},\ee
 where for an extension $L'/L$ of number fields we will denote by $\fd_{L'/L}$ the relative different of $L'$ over $L$ (see Remark \ref{Relative different}). %
Let $\varphi$ be a Hecke character of $K$ with infinity type $\prod_{\up \in \bfa} \frac{|x_{\up}|}{x_{\up}}$ such that the restriction of $\varphi$ to $\bfA_{F}^{\times}$ equals $\chi_{K}$.  Such a character always exists \cite[Lemma A5.1]{ShimuraArithmeticity}.  Set $\psi' = \xi^{-1} \varphi^{-n}$.
Then the theta series, which we (following Shimura) will denote by $\theta_{\xi}$ (so in particular we will suppress $k$ and $\tau$ from notation) belongs to the space $\mS_{n,t+k+n}(\fN_{\ft},\psi')$ (cf. \cite[Proposition A7.16]{ShimuraCBMS97}; cuspidality follows from our assumption that $k_v>0$ \cite[p. 277]{ShimuraCBMS97}).
Here
\be \label{fNt} \fN_{\ft} = (\ft D_{K/F} \Nm_{K/F}(\ff_{\xi})\cap \fd_{F/\bfQ}^{-1} D_{K/F} \cap \fd_{F/\bfQ}^{-1} \ff_{\xi})\fd_{F/\bfQ}
\ee
is an ideal of $\OF$, where $D_{K/F}$ is the relative discriminant of $K$ over $F$ and $$\ft^{-1}:= \{ g^* \tau^{-1} g \mid g \in \OK\}$$  is a fractional ideal of $F$ (cf. \cite[Section A5.5]{ShimuraArithmeticity} and \cite[Prop. A7.16]{ShimuraCBMS97}).

\begin{rem} \label{Relative different} The theta series considered in \cite{ShimuraCBMS97,ShimuraArithmeticity} are of a more general type and are not required to satisfy (\ref{reldiff}). However, in this article we treat only the case of congruence subgroups of the form $\mK_{1,n}(\fM)$ as opposed to the more general kind of $\mK_{1,n}(\fb^{-1}, \fb \fM)$ considered in [loc.cit.]. The assumption (\ref{reldiff}) implies that for the level of the theta series we have $\fb=1$. Furthermore, we will later choose $r$ and $\tau$ (and hence a theta series associated to them) in such a way the $(\tau,r)$-Fourier coefficient $c_f(\tau,r)$ of a form $f \in \mM_{n,k}(\fM)$ (to which we will construct congruences - cf. Theorem \ref{thmmain}) will be an $\ell$-adic unit. However, it follows from \cite[Proposition 18.3(2)]{ShimuraCBMS97} that $c_f(\tau,r)=0$ if $(r^* \tau r)_v \not\in (\fd_{F/\bfQ}^{-1})_v S_n(\OFv)$ for all $v \in \bff$. From this it is easy to show that for a pair $(\tau,r)$ where $c_f(\tau,r) \neq 0$ (and these are the only pairs we are interested in) one has $\{g^* \tau g \mid g \in r\OK^n\} \subset\fd_{F/\bfQ}^{-1}$, hence the assumption (\ref{reldiff}) is often satisfied.
Let us also remark that the purpose of this restriction is to merely simplify the exposition which would require a great deal of additional notation should we consider the general case. However, the method is entirely general and the main result (Theorem \ref{thmmain}) remains true in the case of $\fb \neq 1$ provided one makes the necessary modifications of level of the forms involved. \end{rem}

\begin{prop}[\cite{Klosin15}, Corollary 7.13]\label{prop:thetacoeffs} Let $k, \xi, r, \tau$ be a quadruple as above determining a theta series $\theta_{\xi}$. Assume that $r_{w} = 1_{n}$ for all $w \mid \ff_{\xi} D_{K} \ell$.  For an $\ell$-admissible base $\mB$ there exists a finite extension $E$ of $\bfQ_{\ell}$ so that $e_{\bfa}(-i \tr h_{\bfa}) c_{\theta_{\xi}}(h,b) \in \mO_{E}$ for all $b \in \mB$ and all $h \in S_{n}(F)$.
\end{prop}

\begin{rem} One should note that \cite[Corollary 7.13]{Klosin15} is given for the case $F=\bfQ$, but the results of Shimura it depends on are valid in our case.  The proof given in \cite{Klosin15} then goes through.
\end{rem}

\subsection{Inner product} \label{Inner product}
Let $k,\xi, r,\tau$ and hence $\theta_{\xi}$ be fixed as above. We fix $\fN$ an ideal of $\OF$ with $\fN_{\ft} \mid \fN$ so that $\theta_{\xi} \in \mS_{n, t+k+n}(\fN, \psi')$.
We now give the inner product relation that forms the key to constructing the congruences in this paper.

Set
\begin{equation*}
\Gamma((s)) = (4 \pi)^{-nd\frac{2s + k+ l}{2}} \Gamma_{n}\left(s + (k+ l)/2\right)^{d}
\end{equation*}
where
\begin{equation*}
\Gamma_{n}(s) = \pi^{n(n-1)/2} \prod_{j=0}^{n-1} \Gamma(s - j).
\end{equation*}
Let $Y_{\textrm{re}} = \{h \in \Mat_{n}(\bfC_{\bfa}): h = h^{*}\}/\{h \in \Mat_{n}(\mO_{K}): h = h^{*}\}$ and $Y_{\textrm{im}} = \{h \in \Mat_{n}(\bfC_{\bfa}): h = h^{*}, h> 0\}/\sim$ where $h \sim h'$ if there exists $g \in \GL_{n}(\mO_{K})$ such that $h'= g h g^{*}$.  Then one has $Y_{\textrm{re}} \times Y_{\textrm{im}}$ is commensurable with $\mK_{0,n}(\fN) \cap P_{n}(F) \backslash \bfH_{n}^{\bfa}$,
i.e., the ratio of their volumes is a rational number \cite[p.179]{ShimuraArithmeticity}.

 Set
\begin{equation*}
C(s) = \frac{2^aD_K^{n(n-1)/2}\#X_{m,\fN} h_{K} \Gamma((\ov{s}-n)) (\det \tau)^{-\ov{s} +n - (k+l)/2} |\det r|_{K}^{\ov{s} - n/2} e_{\bfa}(-i \tr \tau_{\bfa}) c_{f}(\tau,r)}{ [\mK_{0,n}(\fN):\mK_{1,n}^{1}(\fN)] \prod_{\up \in \bfc} g_{\up}(\xi(\varpi_{\up}) |\varpi_{\up}|_{\up}^{2\ov{s}+n})}
\end{equation*}
where $\bfc$ is a finite subset of ${\bf f}$ as given in \cite[Proposition 19.2]{ShimuraCBMS97} and $\mK_{1,n}^{1}(\fN) = \{ \kappa \in \mK_{1,n}(\fN): \det \kappa = 1\}$ and $a=1+n(n-1)[F:\bfQ]/2$. Note this differs from \cite[(7.29)]{Klosin15} as the term $h_{K}$ is omitted there while the term $A_N$ there is replaced by $1/(2^{a}D_K^{n(n-1)/2})$ here. We explain this last replacement below (see proof of Theorem \ref{thm:innerproduct}).

For $f_1, f_2 \in \mM_{n,k,\nu}(\mK)$ (with at least one of them a cusp form) we define the inner product $\langle f_1,f_2 \rangle$ as in \cite[(10.9.6)]{ShimuraCBMS97} and set
\be \label{normalizedinner} \langle f_1,f_2 \rangle_{\mK}:=\vol(\mF_{\mK}) \langle f_1,f_2 \rangle,\ee
where $\mF_{\mK}:=(G_n(F)\cap p_b \mK p_b^{-1})\setminus \bfH_n^{\bfa}$ for $b \in \mB$ (we refer the reader to \cite[p. 81]{ShimuraCBMS97} for the definition of the volume and the explanation that the volume is independent of the choice of $b$, but see Remark \ref{volume explicit} where we give a formula for it in the case when $\mK=\mK_{0,n}(\fN)$, the group  which appears in Theorem \ref{thm:innerproduct}). We note here that $\langle f_1, f_2 \rangle$ is independent of the choice of level $\mK$ and satisfies $\langle \rho(g) f_{1}, \rho(g) f_2 \rangle = \langle f_1, f_2 \rangle$ for all $g \in G_{n}(\bfA_{F})$  where $\rho$ is the right regular representation \cite[(10.9.3)]{ShimuraCBMS97}. One can see \cite[(10.9.2a)]{ShimuraCBMS97} for the definition of the inner product for two Hermitian modular forms defined on Hermitian upper half-space.  The key input into our congruence result is the following inner product relation.

\begin{thm}[\cite{Klosin15}, Theorem 7.7]\label{thm:innerproduct} Assume $(h_{K},2n) = 1$
and $\Cl_K^-=\Cl_K$.  Let $f \in \mS_{n,k}(\fN)$ be a Hecke eigenform and let $\xi$ and $\psi'$ be as above.  Then
\begin{equation*}
\langle D(\cdot, s, m, (\psi')^{c}, \fN)\theta_{\xi}, f\rangle_{\mK_{0,n}(\fN)} = \overline{C(s)} \cdot \overline{L(\overline{s}+n/2,f,\xi;\st)}
\end{equation*}
where $C(s)$ is as given above.
\end{thm}
\begin{proof} Let us explain here only the reason for the slight difference between our current definition of $C(s)$ and the definition of $C(s)$ given in \cite[(7.29)]{Klosin15} where $A_N$ got replaced by $1/(2^{1+n(n-1)[F:\bfQ]/2}D_K^{n(n-1)/2})$. The constant $A_N$ is defined on \cite[p. 842]{Klosin15}, but let us denote it here by $A$ as it is denoted on \cite[p. 179]{ShimuraArithmeticity}, where it is defined as a normalization constant between two integrals over ``commensurable'' domains, denoted in \cite{ShimuraArithmeticity} by $\Gamma^P \setminus \mH$ and $X \times Y$.  Using the fact that under our assumptions that $(h_K, 2n)=1$ and $\Cl_K^-=\Cl_K$ we can choose a base $\mB$ consisting of  scalar matrices (Lemma \ref{corl:nicebase}) and use elements of this base as the matrices $q$ in \cite[(22.2a)]{ShimuraArithmeticity} in our case the above domains can be taken to be the following:  $\mK_{0,n}(\fN) \cap P_{n}(F) \backslash \bfH_{n}^{\bfa}$ and $Y_{\rm re} \times Y_{\rm  im}$, where $Y_{\textrm{re}} = \{h \in \Mat_{n}(\bfC_{\bfa}): h = h^{*}\}/\{h \in \Mat_{n}(\mO_{K}): h = h^{*}\}$ and $Y_{\textrm{im}} = \{h \in \Mat_{n}(\bfC_{\bfa}): h = h^{*}, h> 0\}/\sim$ where $h \sim h'$ if there exists $g \in \GL_{n}(\mO_{K})$ such that $h'= g h g^{*}$. It is not difficult to show that the direct product of the fundamental domain for $Y_{\rm re}$ and the fundamental domain for $Y_{\rm im}$ (which is clearly a fundamental domain for $Y_{\rm re} \times Y_{\rm  im}$) is also a fundamental domain for $\mK_{0,n}(\fN) \cap P_{n}(F) \backslash \bfH_{n}^{\bfa}$ after noting that $\mK_{0,n}(\fN) \cap P_{n}(F) \cong \GL_n(\OK) \ltimes \{h \in \Mat_{n}(\mO_{K}): h = h^{*}\}$ with $\GL_n(\OK)$ acting on $\bfH_n^{\bfa}$ by $z \mapsto gzg^*$. Thus the normalization constant can be computed by comparing the formula in \cite[p. 179, line 6]{ShimuraArithmeticity} with \cite[(22.9)]{ShimuraArithmeticity} taking into account differences in measures. More precisely, in our case one has (note that since $F$ is totally real $\nu$ in \cite{ShimuraArithmeticity} equals 2) $A=\frac{2^{n(1-n)[F:\bfQ]}}{2 \vol(Y_{\rm re})}$. Since
$Y_{\rm re} \cong (\bfR^{[F:\bfQ]}/\bfZ^{[F:\bfQ]})^n \times ((K \otimes_{\bfQ} \bfR)/\iota(\OK))^{n(n-1)/2})$ with $\iota: K \to K\otimes_{\bfQ}\bfR$ and $\vol( (K \otimes_{\bfQ} \bfR)/\iota(\OK))=D_K^{1/2}/2^{[F:\bfQ]}$, we get that $A=1/(2^{1+n(n-1)[F:\bfQ]/2}D_K^{n(n-1)/2})$.
\end{proof}

\begin{rem} \label{volume explicit} The definition of the inner product in Theorem \ref{thm:innerproduct} involves the volume of the fundamental domain $\mF_{\mK_{0,n}}(\fN)$. This volume can be computed combining \cite[Theorem 24.4]{ShimuraCBMS97} and \cite[(24.6.5)]{ShimuraCBMS97}. We provide the value here.
\begin{align*}
    \vol(F_{K_{0,n}(\fN)}) &= [\Gamma_{n}(\fN)\cap H:1]2\Nm(\fd_{K/F})^{-n} \left( 2^{2n^2} \pi^{n^2} \prod_{j=1}^{n} \frac{(j-1)!}{(n-j)!} \right)^{[F:\bfQ]} \\
        &\cdot \prod_{j=1}^{2n-1} (j!)^{d} D_{F}^{2n^2 - n} \Nm(\fN)^{4n^2} \prod_{j=1}^{n} \Nm(\fd_{K/F})^{j/2} D_{F}^{1/2} (2\pi)^{-jd} L_{\fN}(j,\chi_{K}^{j}),
    \end{align*}
where $\Gamma_n(\fN) = \mK_{0,n}(\fN) \cap G_n(F)$ and  $H = \{x \in K^{\times}: x \overline{x} =1 \}$.
\end{rem}

\section{Congruence}

For the moment fix a quadruple $k,\xi, r,\tau$ satisfying the conditions of section \ref{sec:theta} and an ideal $\fN$ of $\OF$ such that $\fN_{\ft} \mid \fN$  and let $\theta_{\xi}\in \mS_{n,t+k+n}(\fN, \psi')$ denote the associated theta series as in that section.  We ease the notation in this section by writing $D(g) = D(g,n-m/2,m, (\psi')^{c}, \fN)$
 and $D^{*}(g) = D(g \eta_{\bf f}^{-1}, n - m/2, m, (\psi')^{c}, \fN)$.  Set $l = t+k + n$ as above and put $m = k-l$.
Since $\theta_{\xi} \in \mS_{n,l}( \fN, \psi')$ and $D \in \mM_{n,m}(\fN, (\psi')^{-1})$, we have $D \theta_{\xi} \in \mS_{n,k}(\fN)$ and $(D\theta_{\xi})^{*} \in \mS_{n,k}(\eta_{\bf f}^{-1} \mK_{0,n, {\bf f}}( \fN) \eta_{\bf f})$.

Given $\mK_1 \subset \mK_2$, we define a trace operator
\begin{align*}
\tr_{\mK_1}^{\mK_2}: \mM_{n,k}(\mK_1) &\rightarrow \mM_{n,k}(\mK_2) \\
    f &\mapsto \sum_{\kappa \in \mK_1 \backslash \mK_2} \rho(\kappa) f
\end{align*}
where we recall that $\rho$ is the right regular representation.
If $f$ has $\mO$-integral Fourier coefficients the $q$-expansion principle implies that $\tr f$ also has $\mO$-integral Fourier coefficients \cite[Section 8.4]{HidaSpringer04}.

Now we fix one more ideal $\fM$ of $\OF$ with $\fM  \mid \fN$.
For $g \in G_{n}(\bfA_{K})$, write $\mK^{g} := g^{-1} \mK g$. Set
\begin{equation} \label{defofXi}
\Xi := \tr_{\mK_{0,n}(\fN)^{\eta_{\bf f}} }^{\mK_{0,n}(\fM)^{\eta_{\bf f}}} (\pi^{-nd(n+1)/2}(D\theta_{\xi})^{*}).
\end{equation}

\begin{lemma} \label{FcofXi}  Let $\ell$ be a prime.  Assume $(h_K, 2n)=1$, $\Cl_K^-=\Cl_K$. Further assume that $\ell \nmid D_{K} \Nm_{F/\bfQ}(\fN)(n-2)!$ and that $r$ satisfies $r_w=1_n$ for all $w \mid D_K \ell \textup{cond}(\xi)$. If $\mB$ is an $\ell$-admissible base with support relatively prime to $\cond(\psi)$ there exists a finite extension $E$ of $\bfQ_{\ell}$ so that $e_{\bfa}(- i \tr h_{\bfa}) c_{\Xi}(h,b) \in \mO_{E}$ for all $h \in S_{n}(F)$ and $b \in \mB$.
\end{lemma}

\begin{proof}   The result follows immediately upon combining Proposition \ref{thm:eiscoeffs}, Proposition \ref{prop:thetacoeffs}, and the $q$-expansion principle as explained above.
\end{proof}

\begin{lemma}\label{lem:prods}  For any eigenform $f \in \mS_{n,k}(\fM)$
with $\fM \mid \fN$ we have
\begin{equation*}
\langle \Xi, f^{*} \rangle = \pi^{-nd(n+1)/2} [\mK_{0,n}(\fM): \mK_{0,n}(\fN)] \langle D \theta_{\xi},  f \rangle.
\end{equation*}
\end{lemma}

\begin{proof}
We have
\begin{align*}
 \pi^{nd(n+1)/2} \langle \Xi, f^{*} \rangle &=  \sum_{\kappa \in \mK_{0,n}(\fN)^{\eta_{\bf f}}\backslash \mK_{0,n}(\fM)^{\eta_{\bf f}}} \langle \rho(\kappa) (D\theta_{\xi})^{*}, f^{*} \rangle \\
    &= \sum_{\kappa \in \mK_{0,n}(\fN)^{\eta_{\bf f}}\backslash \mK_{0,n}(\fM)^{\eta_{\bf f}}} \langle (D\theta_{\xi})^{*}, \rho(\kappa^{-1}) f^{*} \rangle \\
    &= [\mK_{0,n}(\fM): \mK_{0,n}(\fN)] \langle (D\theta_{\xi})^{*}, f^{*} \rangle\\
    &= [\mK_{0,n}(\fM): \mK_{0,n}(\fN)] \langle D \theta_{\xi}, f \rangle.
\end{align*}
\end{proof}

We now apply Lemma \ref{lem:prods} together with Theorem \ref{thm:innerproduct} to obtain the following result. For the convenience of the reader we make the statement of Theorem \ref{thmmain} self-contained, repeating all the assumptions and making some of the choices in a slightly different order than in the above narrative.


\begin{thm} \label{thmmain}  Let $F/\bfQ$ be a totally real extension of degree $d$, $K/F$ an imaginary quadratic extension, and assume $\Cl^{-}_{K} = \Cl_{K}$. Let $n$ be a positive integer with $(h_{K},2n) = 1$.
 Let $k = (k, \dots, k) \in \bfZ^{\bfa}$ be a parallel weight so that $k > 0$ and let $\ell$ be a rational prime so that $\ell > k$ and $\ell \nmid D_{K} h_{K}$.  Let $f \in \mS_{n,k}(\fM)$ be a Hecke eigenform with $\mO$-integral Fourier coefficients for $\mO$ the ring of integers in some algebraic extension $E/\bfQ_{\ell}$.  
Let $\xi$ be a Hecke character of $K$ so that
$\xi_{\bfa}(z) = \left(\frac{z}{|z|}\right)^{-t}$ for some $t = (t, \dots, t) \in \bfZ^{\bfa}$ with $-k \leq t < \min\{-6, -4n\}$. 
 Assume there exists
$\tau \in S^+_{n}(\OF)$
and $r \in \GL_{n}(\bfA_{K,{\bfk}})$ 
 such that $r_w=1_n$ for all $w \mid D_K\ell \textup{cond}(\xi)$,  the condition (\ref{reldiff}) holds and $\val_{\varpi}(e_{\bfa}(- i \tr h_{\bfa})c_{f}(\tau,r)) = \val_{\varpi}(\det r)= 0$.
Let $\fN$ be an ideal of $\OF$ such that $\fN_{\ft} \mid \fN$ and $\fM \mid \fN$ with $\fN_{\ft}$ defined as in (\ref{fNt}). Then there exists a field subextension $\bfQ_{\ell} \subset E'\subset E$, finite over $\bfQ_{\ell}$ with uniformizer $\varpi$, containing  the algebraic value \begin{equation*}
L^{\rm alg}(2n+t/2, f, \xi;\st) = \frac{L(2n+t/2, f, \xi; \st)}{\pi^{n(k + 2n + t +1)d} \langle f, f \rangle}
\end{equation*} such that if
$$\val_{\varpi}\left( \#(\mO_{K}/\fN \mO_{K})\#(\mO_{K}/\fN\mO_{K})^{\times}[\mK_{0,n}(\fM): \mK_{0,n}(\fN)]\right)=0$$
and
\begin{equation*}
-b = \val_{\varpi}\left(\frac{\pi^{dn^2}}{\vol(\mF_{\mK_{0,n}(\fM)})}\ov{L^{\rm alg}(2n+t/2, f, \xi; \st)}\right)  < 0,
\end{equation*}
then there exists $f' \in \mS_{n,k}(\fM)$, orthogonal to $f$, with $\Oo$-integral Fourier coefficients  so that $f \equiv f' \pmod{\varpi^{b}}$.
\end{thm}

\begin{proof}  We begin by noting that \cite[Theorem 28.8]{ShimuraArithmeticity} gives that $L^{\rm alg}(2n+t/2, f, \xi;\st) \in \ov{\bfQ}$ and \cite[Theorem 24.7]{ShimuraCBMS97} gives that $\pi^{dn^2}/\vol(\mF_{\mK_{0,n}(\fM)}) \in \ov{\bfQ}$.  Note here that \cite[Theorem 24.7]{ShimuraCBMS97} is stated for anisotropic forms, but one can easily show $G_{n}(F_{\bfa})$ is isomorphic as a Lie group to the unitary group of an anisotropic form so it applies in our case as well.

For the quadruple $k, \xi, r, \tau$ we define the theta series $\theta_{\xi}$ as in section \ref{sec:theta}. The assumption that $\fN_{\ft} \mid \fN$ guarantees that $\theta_{\xi} \in \mS_{n,t+k+n}(\fN, \psi')$. We also attach the Eisenstein series $D=D(g, n-(k-l)/2, (\psi')^c, \fN)$ to the ideal $\fN$ and the character $\psi'$. We then define $\Xi$ as in (\ref{defofXi}).
 By Lemma \ref{FcofXi} our assumptions guarantee that $\Xi$ has $\Oo'$-integral Fourier coefficients for $\Oo'$ the ring of integers in some finite extension $E'$ of $\bfQ_{\ell}$. We will write $\varpi$ for a uniformizer of $E'$. By extending $E'$ (and thus $\Oo'$) if necessary we may also assume that $L^{\rm alg}(2n+t/2, f, \xi;\st)$ and  $\pi^{dn^2}/\vol(\mF_{\mK_{0,n}(\fM)})$ both lie in $E'$.
 Indeed, note that here we are forcing $\ell > n$, hence the assumption that $\val_{\varpi}((n-2)!)=0$ in Lemma \ref{FcofXi} is unnecessary.
 Write
\begin{equation} \label{fir}
\Xi = C_{f} f^{*} + g^{*}
\end{equation}
for some $g \in \mS_{n,k}(\fM)$ where $\langle f, g \rangle = 0$ and $C_{f} = \frac{\langle \Xi, f^{*} \rangle}{\langle f^{*}, f^{*} \rangle}$.

We will be  interested in the $\ell$-adic valuation of $C_f$
so we do not require the exact value. Lemma \ref{lem:prods} and Theorem \ref{thm:innerproduct} give
\begin{align*}
\langle \Xi, f^{*}\rangle &= \pi^{-nd(n+1)/2} [\mK_{0,n}(\fM):\mK_{0,n}(\fN)] \langle D \theta_{\xi}, f \rangle\\
    &= \frac{[\mK_{0,n}(\fM):\mK_{0,n}(\fN)]}{\pi^{nd(n+1)/2}  \vol(\mF_{\mK_{0,n}(\fM)})} \langle D \theta_{\xi}, f \rangle_{\mK_{0,n}(\fM)} \\
    &= \frac{[\mK_{0,n}(\fM):\mK_{0,n}(\fN)]}{\pi^{nd(n+1)/2}  \vol(\mF_{\mK_{0,n}(\fM)})} \ov{C((3n+t)/2)} \cdot \ov{L(2n+t/2, f, \xi;\st)}.
\end{align*}
It now remains to simplify $C((3n+t)/2)$.  We have
\begin{equation*}
C((3n+t)/2) = \frac{u \#X_{-t-n,\fN} h_{K} \Gamma((\frac{n+t}{2})) (\det \tau)^{-(n+t+k)} |\det r|_{K}^{n+t/2} e_{\bfa}(-i \tr \tau_{\bfa}) c_{f}(\tau,r)}{[\mK_{0,n}(\fN):\mK_{1,n}^{1}(\fN)] \prod_{\up \in \bfc} g_{\up}(\xi(\varpi_{\up}) |\varpi_{\up}|_{\up}^{4n+t})}
\end{equation*}
for $u$ a $\varpi$-adic unit.
Note here the fact that $l = t + k +n$ has been used in simplifying this expression.
We have
\begin{equation*}
\Gamma\left(\left(-\frac{t+n}{2}\right)\right) = (4\pi)^{-n d(t+k+n)} \pi^{nd(n-1)/2} \prod_{j=0}^{n-1} \Gamma(t+k+n-j)^{d}.
\end{equation*}
Note the largest value of the argument of $\Gamma$ in the above product is $t+k+1$. By our assumption on the allowable range of the values of $t$ we see that this value is never greater than $k$, which is less than $\ell$. Hence the product of the $\Gamma$-factors is a $\varpi$-adic unit. Observe also that $\prod_{\up \in \bfc} g_{\up}(\xi(\varpi_{\up}) |\varpi_{\up}|_{\up}^{4n+t})$ is a finite product and the $g_{\up}$ are polynomials with coefficients in $\bfZ$ and a constant term of 1 \cite[Lemma 20.5]{ShimuraArithmeticity}.
Thus  as long as $-t > 4n$ we have this lies in a finite extension of $\bfZ_{\ell}$. Thus, (extending $E'$ if necessary) we get  $\val_{\varpi}(1/\prod_{\up \in \bfc} g_{\up}(\xi(\varpi_{\up}) |\varpi_{\up}|_{\up}^{4n+t})) \leq 0$.  Moreover, $[\mK_{0,n}(\fN):\mK_{1,n}^{1}(\fN)] \in \bfZ$, so $\val_{\varpi}(1/[\mK_{0,n}(\fN):\mK_{1,n}^{1}(\fN)]) \leq 0$.  We also have
 $\det \tau \in \fd_{K/F}^{-1}$.
However, since $\ell \nmid D_K$ and $-(n+t+k) <0$
 we have (assuming $K \subset E'$) that  $\val_{\varpi}((\det \tau)^{-(n+t+k)}) \leq 0$. Since $\fN \neq (1)$ the proof of \cite[Lemmas 11.14, 11.15]{ShimuraCBMS97} (and the remarks that follow these lemmas) give that $\# X_{-t-n, \fN}$ equals the index of the group
$\{x \in \bfA_{K}^{\times}: \textrm{$x_{w} \in \mO_{K,w}^{\times}$ and $x_{w} - 1 \in \fN \mO_{K,w}$ for every $w \in \bfk$}\}$
 in $\bfA_{K}^{\times}$.  Thus, $\val_{\varpi}(\# X_{-t-n,\fN}) = 0$.
Finally, note that $|\det r|_K$ as well as $h_K$ are also $\varpi$-adic units by our assumptions.
We can now conclude that
\begin{equation*}
\frac{\langle \Xi, f^{*} \rangle}{\langle f^{*}, f^{*} \rangle} = (*) \left(\frac{\pi^{dn^2}}{\vol(\mF_{\mK_{0,n}(\fM)})} \ov{L^{\rm alg}(2n+t/2, f, \xi;\st)}\right)
\end{equation*}
where  $(*) \in E'$ with $\val_{\varpi}(*) \leq 0$ and we have used that $\langle f, f \rangle = \langle f^{*}, f^{*} \rangle$.
  As the $\varpi$-valuation of $\frac{\pi^{dn^2}}{\vol(\mF_{\mK_{0,n}(\fM)})} \ov{L^{\rm alg}(2n+t/2, f, \xi;\st)}$ is assumed to be $-b < 0$ we see there exists some positive integer $a \geq b$ so that $C((3n+t)/2) = u \varpi^{-a}$ with $u \in (\mO')^{\times}$.  Thus, we have
\begin{equation}\label{almostfinal}
u f^{*} = \varpi^{a} \Xi - \varpi^{a} g^{*}.
\end{equation}
Since the Fourier coefficients of $f$ are $\Oo$-integral by assumption, so are the Fourier coefficients of $f^*$ by the $q$-expansion principle.
As the Fourier coefficients of $\Xi$ are $\mO'$-integral
 and thus also $\Oo$-integral
 by Lemma \ref{FcofXi},
we get that  the Fourier coefficients of $\varpi^a g^*$ are $\Oo$-integral.
Rearranging (\ref{almostfinal}) and using the $\Oo'$-integrality of the Fourier coefficients of $\Xi$ we get that the Fourier coefficients of $uf^* + \varpi^a g^*$ are $\Oo'$-integral and (since $u \in (\Oo')^{\times}$) we get
\begin{equation*}
f^{*} \equiv -u^{-1} \varpi^{a} g^{*} \pmod{\varpi^{a}}.
\end{equation*}
Thus, we obtain
\begin{equation}\label{ReadyforT}
f \equiv -u^{-1} \varpi^{a} g \pmod{\varpi^{a}}
\end{equation}
and $-u^{-1} \varpi^{a} g$ can be taken as the form $f'$ in the statement of the theorem.
\end{proof}

\begin{rem} \label{couldbezero} The form $f'$ in the statement Theorem \ref{thmmain} can a priori be zero. However, it is non-zero whenever $f \not\equiv 0 \pmod{\varpi^b}$. This follows immediately from (\ref{ReadyforT}). \end{rem}


\section{Arithmetic properties of Ikeda lifts} \label{Arithm}

For the rest of the paper we restrict our attention to the case when $F =\bfQ$ and $K= \bfQ(\sqrt{-D_K})$ is an imaginary quadratic extension of $\bfQ$ with discriminant $-D_{K}$. We write $\bfA$ for $\bfA_{\bfQ}$. In this section we will show that the Fourier coefficients of Ikeda lifts (with respect to some base) are integral, generate a number field and we formulate a condition on a certain (mod $\ell$) Galois representation which will ensure they are also non-vanishing mod $\ell$. This will provide a complementary result to Ikeda's non-vanishing result (cf. Theorem \ref{nonzero Ikeda}).

\subsection{Generalities on Ikeda lifts}
In this context we will need the base change and symmetric square $L$-functions as well; we recall the definitions here. For positive integers $k,N$ and a Dirichlet character $\chi: (\bfZ/N\bfZ)^{\times} \to \bfC^{\times}$ we will denote by $S_k(N, \chi)$ the space of (classical) elliptic cusp forms of weight $k$, level $N$ and nebentypus $\chi$. If $\chi$ is trivial we will omit it. 
For $p \nmid N$ and $\phi \in S_{k}(N,\chi)$ a primitive eigenform, let $\alpha_{\phi,p},\beta_{\phi,p}$ be the $p$-Satake parameters of $\phi$ normalized arithmetically, i.e., so that $\alpha_{\phi,p} \beta_{\phi,p} = p^{k-1} \chi(p)$.
For a Dirichlet character $\psi : (\bfZ/N\bfZ)^{\times} \to \bfC^{\times}$ and  $s \in \bfC$ with $\textrm{Re}(s)$ sufficiently large the (partial) symmetric square $L$-function is defined by
\begin{equation*}
L^{N}(s, \Sym^2 \phi\otimes \psi) = \prod_{p\nmid N} \left[ (1 - \alpha_{\phi,p}^2 \psi(p)p^{-s})(1 - \alpha_{\phi,p} \beta_{\phi,p}\psi(p) p^{-s})(1 - \beta_{\phi,p}^2\psi(p) p^{-s})\right]^{-1}.
\end{equation*}
If $N$ or $\psi$ are $1$ we drop them from notation.

We define the (twisted) base change $L$-function from $\bfQ$ to $K$ as follows. Let $\psi$ be a Hecke character of $K$ of conductor dividing $N$.  For a place $w$  of $K$ of residue characteristic $p \nmid N$, set $\alpha_{\phi,w} = \alpha_{\phi,p}^{d}$ and $\beta_{\phi,w} = \beta_{\phi,p}^{d}$ where $d = [\mO_{K,w}/\varpi_w \Oo_{K, w}: \bfF_{p}]$ and $\varpi_w$ denotes a uniformizer of $K_w$.  For $s \in \bfC$ with $\textrm{Re}(s)$ sufficiently large the (partial) base change $L$-function is defined by
\begin{equation*}
L^{N}(s,\BC(\phi)\otimes \psi) = \prod_{w \nmid N} [(1 - \alpha_{\phi,w} \psi(\varpi_{w})  |\varpi_{w}|_{w}^{s})(1 - \beta_{\phi,w} \psi(\varpi_{w})  |\varpi_{w}|_{w}^{s})]^{-1}.
\end{equation*}
Here again we identify $\varpi_w$ with its image in $\AK^{\times}$.




Let $n=2m$ (resp. $n=2m+1$). Let $\phi$ be a newform in $S_{2k+1}(D_K, \chi_K)$ (resp. in $S_{2k}(1)$).
Ikeda has shown (\cite[Section 5]{Ikeda08}, see also \cite[Theorems 2.1, 2.2]{Katsurada17} and note that the product of the $L$-functions in [loc.cit.] agrees with the base change $L$-function below) that there exists a Hecke  eigenform $I_{\phi} \in \mS_{n, 2k+2m, -k-m}(G_n(\hat{\bfZ}))$ such that for any Hecke character $\psi$ of $K$ we have
\be \label{num1}
L^{D_K}(s, I_{\phi}, \psi; {\rm st}) = \prod_{i=1}^n L^{D_K}(s+k+m-n-i+1, \BC(\phi)\otimes \psi).
\ee
\begin{rem} \label{class vs adelic} The reader will note that we use the classical and adelic language somewhat inconsistenly reserving the first one for the elliptic modular form $\phi$ while using the second one for its Ikeda lift $I_{\phi}$. This is however the convention used by both Ikeda \cite{Ikeda08} and Katsurada \cite{Katsurada17} and we chose not to alter it here in order to make the references which we use more readily applicable to our situation. \end{rem}

\begin{rem} \label{discrepancy} In \cite{Katsurada17}  an automorphic normalization of the standard $L$-function is used, which we temporarily denote by $L_0(s)$. However, 
here and in the rest of the paper we use Shimura's normalization. The translation is given by $L(s)=L_0(s-n+1/2)$. In particular in the normalization given in \cite{Katsurada17} in (\ref{num1}) one has factors of the form $L_0^{D_K}(s+k+m-i+1/2, \BC(f)\otimes \psi)$ which in our normalization equals $L^{D_K}(s+k+m-n-i+1,  \BC(f)\otimes \psi)$. We also note that our use of letters $n$ and $m$ is reversed from that in \cite{Katsurada17}. \end{rem}

For $\phi$ as above let $$N_{\phi}:= \begin{cases} D_K & \textup{if $n=2m$}\\ 1 & \textup{if $n=2m+1$}\end{cases} \quad \textup{and} \quad \Gamma_{\phi}:= \Gamma_0(N_{\phi})$$
If $\phi'$ is another form of  level $\Gamma_1(N_{\phi})$ of the same weight $w$ as $\phi$ and the same nebentypus  we set
$$
\langle \phi,\phi' \rangle_{\Gamma_{\phi}}:= \int_{\Gamma_{\phi} \setminus\bfH} \phi(z)\ov{\phi'(z)}y^{w-2}dx \hspace{2pt} dy,$$ and $$\langle \phi,\phi' \rangle:= \frac{1}{i(\phi)}\langle \phi,\phi' \rangle_{\Gamma_{\phi}},$$ where $\bfH$ is the complex upper half-plane, $x=\textup{Re} (z)$, $y=\textup{Im}(z)$, $i(\phi)=[\ov{\SL_2(\bfZ)}: \ov{\Gamma_{\phi}}]$, $\ov{\SL_2(\bfZ)}:= \SL_2(\bfZ)/\langle -1_2\rangle$ and $\ov{\Gamma}_{\phi}$ is the image of $\Gamma_{\phi}$ in $\ov{\SL_2(\bfZ)}$.

\begin{rem} \label{power of pi} Shimura \cite{ShimuraMathAnn77} and Sturm \cite{Sturm} (whose results we will later use) define the inner product differently by dividing the integral by the volume of the fundamental domain. If we denote the inner product used in \cite{ShimuraMathAnn77,Sturm} by $\langle\cdot, \cdot\rangle_S$, then one has $\langle\cdot, \cdot\rangle = \frac{\pi}{3} \langle\cdot, \cdot\rangle_S$.
\end{rem}

It has been proven by Katsurada \cite[Theorem 2.2]{Katsurada17} that one has
\begin{multline} \label{den1} \left< I_{\phi}, I_{\phi} \right>_{G_n(\hat{\bfZ})}=(*)\left<\phi, \phi\right>_{\Gamma_{\phi}}  \times\\ \begin{cases}\prod_{i=2}^n  L(i+2k-1, \Sym^2 \phi \otimes \chi_K^{i+1})L(i, \chi_K^i)\Gamma_{\bfC}(i+2k-1)\Gamma_{\bfC}(i)^2,&n=2m+1\\
\prod_{i=2}^n  L(i+2k, \Sym^2 \phi \otimes \chi_K^{i})L(i, \chi_K^i)\Gamma_{\bfC}(i+2k)\Gamma_{\bfC}(i)^2,&n=2m\end{cases}\end{multline}
where $(*)$ is an integer divisible only by powers of $2$ and $D_{K}$, 
 and $\Gamma_{\bfC}(s):= 2(2\pi)^{-s} \Gamma(s)$.


\subsection{Integrality of the Fourier coefficients of  Ikeda lift}

Let $\phi$ be as above and write $I_{\phi}$ for its Ikeda lift to $G_n(\AQ)$.

\begin{prop} \label{integers1} Let $\mB$ be an admissible base. Then for every $b \in \mB$ and $h \in S_n(\bfQ)$ 
the normalized Fourier coefficient $e(-i \tr h) c_{I_{\phi}}(h, p_b)$ is an algebraic integer.
\end{prop}

\begin{proof}
Let $\mB$ be an admissible base. For $b \in \mB$ the formulas of \cite[p. 7]{Katsurada17} give that
\be \label{Fc formula}
e(-i \tr h)c_{I_{\phi}}(h, p_b) = |\gamma(h)|^x\prod_{p \mid \gamma(h)}\tilde{F}_p(h, \alpha_{\phi,p}').
\ee
Here $x=k$ if $n=2m$ (resp. $x=k-1/2$ if $n=2m+1$), $\gamma(h) = (-D_{K})^{\lfloor n/2\rfloor}\det h\in \bfZ$, $\tilde{F}_p(h,X)$ is a Laurent polynomial with coefficients in $\bfZ$ whose top degree term is $X^{\val_p(\gamma(h))}$ and the bottom degree term is a root of unity times $X^{-\val_p(\gamma(h))}$ (cf. \cite[p.1112]{Ikeda08}). Finally, we set
\be
\alpha_{\phi, p}':= \begin{cases} p^{-k}\alpha_{\phi,p} & n=2m\\
p^{-k+1/2}\alpha_{\phi,p} & n=2m+1,\end{cases}
\ee
where $\alpha_{\phi,p}$  is any Satake parameter of $\phi$ at $p$ if $n=2m+1$ or if $n=2m$ and $p\nmid D_K$ and equals the $p$th Fourier coefficient of $\phi$ if $n=2m$ and $p \mid D_K$.
As noted by Ikeda \cite[p. 1118]{Ikeda08} the independence of $\tilde{F}_p(h, \alpha'_{\phi, p})$ from the choice of a Satake parameter at $p$ follows from the functional equation satisfied by $\tilde{F}_p(h, X)$ \cite[Lemma 2.2]{Ikeda08}.
 The above formula shows that $e(-i \tr h)c_{I_{\phi}}(h,p_b)$ lies in some number field $K'$, hence it remains to show that its $\fp$-adic valuation is non-negative for each prime $\fp$ of $K'$.  Let us only show the claim in the case when $\fp$ lies over $p \nmid D_K$ and $n=2m$ (the other cases are handled analogously).  Suppose $p \nmid D_K$, $n=2m$ and set $y:= \val_p(\gamma(h))\geq 0$. Then $|\gamma(h)|^k = p^{yk}u$ for some $u\in \bfZ$ with $p \nmid u$.
We will prove that $\val_{\fp}(\tilde{F}_{l}(h, \alpha'_{\phi,l})) \geq 0$ for all $l \mid \gamma(h)$ with $l\neq p$ and that $\val_{\fp}(|\gamma(h)|^k\tilde{F}_p(h, \alpha'_{\phi,p}))\geq 0$.
First consider the case of $l\neq p$ with $\ell \mid \gamma(h)$. Since $\alpha_{\phi, l}\beta_{\phi,l}=\chi_K(l)l^{2k}$, where $\beta_{\phi,l}$ stands for the other $l$-Satake parameter of $\phi$, we must have that $\val_{\fp}(\alpha'_{\phi,l})=0$. The claim now follows from the fact that $\tilde{F}_l(h,X)$ has coefficients in $\bfZ$. Hence it remains to consider that case of the prime $p$.
For simplicity write $\alpha$ for $\alpha_{\phi,p}$ and $\beta$ for the other $p$-Satake parameter. Write $e$ for the ramification index of $\fp$ over $p$.
It suffices to show that both $V_1:=\val_{\fp}(|\gamma(h)|^k(p^{-k} \alpha)^{\val_p(\gamma(h))})\geq 0$ and $V_2:=\val_{\fp}(|\gamma(h)|^k (p^{-k}\alpha)^{-\val_p(\gamma(h))})\geq0$.
A direct calculation yields
$V_1=eyk+y(-ke+\val_{\fp}(\alpha)) = y\val_{\fp}(\alpha)$ and $V_2 = eyk-y(-ke+\val_{\fp}(\alpha)) = y(2ek-\val_{\fp}(\alpha)).$
 Note that since $\alpha$ and $\beta$ are algebraic integers, we have $\val_{\fp}(\alpha)\geq0$ and $\val_{\fp}(\beta)\geq 0$. This alone shows that $V_1 \geq 0$. We also have $\alpha \beta = \chi_K(p) p^{2k}$, which yields $\val_{\fp}(\alpha) +\val_{\fp}(\beta) = 2ke$ and hence $\val_{\fp}(\alpha) \leq 2ke$, which implies that $V_2 \geq 0$. \end{proof}

\begin{cor}\label{one number field} Let $\mB$ be an admissible base. Let $L$ be a number field  containing all the Fourier coefficients of the newform $\phi$. Write $\OL$ for the ring of integers of $L$. Then for all $b \in \mB$ and $h \in S_n(\bfQ)$  one has $e(-i \tr h) c_{I_{\phi}}(h, p_b) \in \OL$.
\end{cor}

\begin{proof} First note that it is a well-known fact that the Fourier coefficients of any newform are all contained in finite extension of $\bfQ$.  By Lemma \ref{integers1} it is enough to show that $e(-i \tr h) c_{I_{\phi}}(h, p_b) \in L$ for all $b \in \mB$ and $h \in S_n(\bfQ)$. Fix such a pair $h,b$. Then it suffices to show that $\tilde{F}_{p}(h, \alpha'_{\phi,p}) \in L$ for all $p \mid \gamma(h)$. Since $\alpha'_{\phi,p}\in L'$, where $L'=L$ or $L'/L$ is a quadratic extension, we get that $\tilde{F}_{p}(h, \alpha'_{\phi,p})\in L'$. If $L'=L$ we are done. Otherwise let $\sigma$ be the non-trivial element of $\Gal(L'/L)$ and write $\beta'_{\phi, p}$ for the other $p$-Satake parameter of $\phi$ (normalized as $\alpha'_{\phi,p}$). Then one clearly has $\sigma(\alpha'_{\phi,p}) = \beta'_{\phi,p}$ and thus $\sigma(\tilde{F}_{p}(h, \alpha'_{\phi,p})) = \tilde{F}_{p}(h, \beta'_{\phi,p})= \tilde{F}_{p}(h, \alpha'_{\phi,p})$, where the last equality follows from the independence of $\tilde{F}_{p}(h, \alpha'_{\phi,p})$ of the choice of a particular $p$-Satake parameter (see proof of Lemma \ref{integers1}). Hence the claim follows.   \end{proof}

\begin{cor} \label{cong between phis} Let $\phi, \phi'$ be two newforms in $S_{2k+1}(D_K, \chi_K)$ (if $n=2m$) or in $S_{2k}(1)$ (if $n=2m+1)$. Let $\fp$ be a prime of a number field $L$ containing the Fourier coefficients of both $\phi$ and $\phi'$. Suppose that $\phi \equiv \phi'$ (mod $\fp^r$). Then $I_{\phi} \equiv I_{\phi'}$ (mod $\fp^r$).
\end{cor}

\begin{proof} Since $\phi, \phi'$ are newforms, their Hecke eigenvalues are congruent mod $\fp^r$. From this it is easy to see that for every prime $p$ not dividing the level, the corresponding $p$-Satake parameters (which are integral over $L$) are also congruent mod $\fp^r$. Then it follows from the proof of Lemma \ref{integers1} that $\tilde{F}_p(h, \alpha'_{\phi, p})$ (which live in $\OL$ by Lemma \ref{one number field}) are congruent if $\fp \nmid p$ and $|\gamma(h)|^x \tilde{F}_p(h, \alpha'_{\phi, p})$ are congruent if $\fp \mid p$. The case of $p$ dividing the level is immediate. \end{proof}

\subsection{Non-vanishing of the Fourier coefficients of Ikeda lift mod $\ell$}

The following result is due to Ikeda.

\begin{thm} \label{nonzero Ikeda} The Ikeda lift $I_{\phi}$ is non-zero unless $n\equiv 2$ (mod 4) and $\phi$ arises from a Hecke character of some imaginary quadratic field.
\end{thm}

\begin{proof} This follows from Corollary 14.2 and Corollary 15.21 of \cite{Ikeda08}. \end{proof}

In this section we prove a mod $\ell$-version of this result. Write $\ov{\rho}_{\phi}: G_{\bfQ} \to \GL_2(\ov{\bfF}_{\ell})$ for the semi-simple
 residual Galois representation attached to $\phi$.

\begin{thm} \label{not cong to zero}  Let $\ell\nmid 2D_K$ be a prime. Let $\mB$ be an admissible base. If $n \not \equiv 2$ (mod 4), then with respect to $\mB$ at least one of the Fourier coefficients of $I_{\phi}$ has $\ell$-adic valuation equal to zero. If $n \equiv 2$ (mod 4) and $\ov{\rho}_{\phi}(G_K)$ is a non-abelian subgroup of $\GL_2(\ov{\bfF}_p)$, then 
  with respect to $\mB$ at least one of the Fourier coefficients of $I_{\phi}$ has $\ell$-adic valuation equal to zero. \end{thm}

\begin{rem} \label{Ribetcondition} A connection between a modular eigenform $\phi$  arising from a Hecke character of an imaginary quadratic field $K'$ and the condition that its $\ell$-adic (so, in particular characteristic zero) Galois representation $\rho_{\phi}$ has abelian image when restricted to $G_{K'}$ was proved by Ribet  \cite{RibetModFuncOneVar77}. This combined with Theorem \ref{nonzero Ikeda} above yields a connection between non-vanishing of the Ikeda lift $I_{\phi}$ and the condition that  $\rho_{\phi}(G_K)$ be non-abelian.  Our result in essence provides an analogous connection on a mod $\ell$ level. \end{rem}

\begin{proof} [Proof of Theorem \ref{not cong to zero}]
We pick an admissible base $\mB$.  If $n \not \equiv 2$ (mod 4), then the  assertion follows from \cite[Lemmas 11.1 and 11.2]{Ikeda08} because they ensure that there exists $h$ for which $\gamma(h)=1$ and thus we get $e(-i \tr h) c_{I_{\phi}}(h, p_b)=1$ for all $b \in \mB$. Hence for the rest of the proof we assume $n\equiv 2$ (mod 4). We will prove Theorem \ref{not cong to zero} by a sequence of lemmas.

\begin{lemma} \label{inert1} Suppose $p$ is an odd prime such that $p \equiv -1 \pmod{D_K}$. Then $p$ is inert in $K$.
\end{lemma}

\begin{proof} Note that $K \subset \bfQ(\zeta_{D_K})$ and that the composite of the canonical homomorphisms $(\bfZ/D_K \bfZ)^{\times} \xrightarrow{\sim} \Gal(\bfQ(\zeta_{D_K})/\bfQ) \twoheadrightarrow \Gal(K/\bfQ) \xrightarrow{\sim} \{\pm 1\} \subset \bfC^{\times}$ equals $\chi_K$. The Lemma follows immediately from this fact.
\end{proof}

\begin{lemma} \label{equal Fc}  For every odd prime $p$ with $p \equiv -1$ \textup{(mod $D_K$)} there exists a matrix $h_{p}$ such that  $e(-i \tr h_{p})c_{I_{\phi}}(h_{p}, p_b)=a_{\phi}(p)$ for every $b \in \mB$.  Here $a_{\phi}(p) = \alpha_{\phi,p} + \beta_{\phi,p}$ denotes the $p$th Fourier coefficient of $\phi$.\end{lemma}

\begin{proof} By \cite[Lemma 11.4]{Ikeda08} and its proof for every such $p$ there exists a matrix $h_p$ such that $\gamma(h_p)=-p$, where $\gamma$ is as in (\ref{Fc formula}). Hence by that same formula, we get that $$e(-i \tr h_{p})c_{I_{\phi}}(h_p, p_b) = |\gamma(h_p)|^k\tilde{F}_p(h_p, \alpha'_{\phi, p}).$$
By \cite[p. 1112]{Ikeda08} we know that $\tilde{F}_p(h_p,X)$ is a Laurent polynomial whose highest degree term is $X$ and whose lowest degree term is $\underline{\chi}_p(\gamma(h_p))^{n-1}X^{-1}$, i.e., $\tilde{F}_p(h_p,X) = X + a +\underline{\chi}_p(\gamma(h_p)) 1/X$ for some integer $a$. Here  $\underline{\chi}_p(a):= \left(\frac{-D_K,a}{\bfQ_p}\right)$ (cf. \cite[p. 1110]{Ikeda08}), where the latter denotes the Hilbert symbol. Using the assumptions on $p$, we get $\underline{\chi}_p(\gamma(h_p))= -1$.
 By \cite[Lemma 2.2]{Ikeda08} we have that the functional equation for $\tilde{F_p}$ reads $$\tilde{F}_p(h_p; 1/X) = \underline{\chi}_p(\gamma(h_p))\tilde{F}_p(h_p;X).$$
So, we have $$
1/X+a-X =\tilde{F}_p(h_p; 1/X)=-\tilde{F}_p(h_p,X)= (-1)(X+a-1/X).$$ From this it follows that $\tilde{F}_p(h_p,X)=X-1/X$.

Thus $$e(-i \tr h_{p})c_{I_{\phi}}(h_p, p_b)=|\gamma(h_p)|^k\tilde{F}_p(h_p, p^{-k}\alpha_{\phi,q}) = p^k\left(\alpha_{\phi,p}p^{-k}-\frac{1}{\alpha_{\phi,p}p^{-k}}\right).$$ Since $\alpha_{\phi,p}\beta_{\phi,p}=p^{2k}\chi_K(p) = -p^{2k}$, we get $$e(-i \tr h_{p})c_{I_{\phi}}(h_p, p_b)= \alpha_{\phi,p}+\beta_{\phi,p} = a_p(\phi).$$
\end{proof}

As already noted  one has $K \subset \bfQ(\zeta_{D_K})$. Write $L$ for the splitting field of $\ov{\rho}_{\phi}$.

\begin{lemma} \label{intersection} One has $L \cap \bfQ(\zeta_{D_K}) = K$.
\end{lemma}

\begin{proof} For simplicity set $D=D_K$. For any $p \mid D$ and any prime $\fp$ of $L$ lying over $p$, write $I_{\fp}$ for the inertia subgroup of $\fp$. One has that \be\label{when res} \ov{\rho}_{\phi}|_{I_{\fp}} \cong 1 \oplus \chi'_K,\ee where $\chi'_K: \Gal(\bfQ(\zeta_{D})/\bfQ) \to \ov{\bfZ}_{\ell}^{\times}$ is the Galois character obtained from $\chi_K: (\bfZ/D\bfZ)^{\times}\to \bfC^{\times}$ via the canonical identification $\Gal(\bfQ(\zeta_{D})/\bfQ) \cong (\bfZ/D\bfZ)^{\times}$ (\cite[Theorem 3.1(e)]{DarmonDiamondTaylor95} or \cite[Theorem 3.26(3)]{HidaGaloisCohomology}).

Set $K'=L\cap \bfQ(\zeta_D)$. One has the following tower of fields $\bfQ \subset K \subset K' \subset L$ and
\eqref{when res} implies that the extension $L/K'/K$ is unramified outside the primes above  $\ell$. On the other hand, since $K' \subset \bfQ(\zeta_D)$, the group $\Gal(K'/K)$ is generated
by images  of the inertia subgroups $I_{\fp}$ for $\fp$ lying over $p$ with $p \mid D$. We conclude that $K' = K$.
\end{proof}

 Let $\varpi$ be a uniformizer of a finite extension $E$ of $\bfQ_{\ell}$ in which all the Fourier coefficients of $\phi$ lie. Suppose that $e(-i \tr h)c_{I_{\phi}}(h, p_b)\equiv 0 \pmod{\varpi}$ for all matrices $h$ (this makes sense by Corollary \ref{one number field}). By Lemma \ref{equal Fc} this implies that $a_{\phi}(p) \equiv 0 \pmod{\varpi}$ for all odd primes $p \equiv -1 \pmod{D_K}$. Since for every prime $p \nmid \ell D_K$ we have $\tr \rho_{\phi} (\Frob_p) = a_p(\phi)$, we see that we must have $\tr \ov{\rho}_{\phi}(\Frob_p) \equiv 0 \pmod{\varpi}$ for all odd primes $p \nmid \ell D_K$ with $p \equiv -1 \pmod{D_K}$.

\begin{lemma}\label{Lemma 1} One has $\Gal(L/\bfQ) = G \sqcup cG$, where $$G:=\{\textup{conjugates of }\Frob_p \mid p \equiv 1 \pmod{D_{K}}\}=\Gal(L/K)$$ and $cG=\{\textup{conjugates of }\Frob_p \mid p \neq 2, p \equiv -1 \pmod{D_{K}}\}.$ Here $c$ is complex conjugation. \end{lemma}

\begin{proof} Consider the following diagram of fields $$
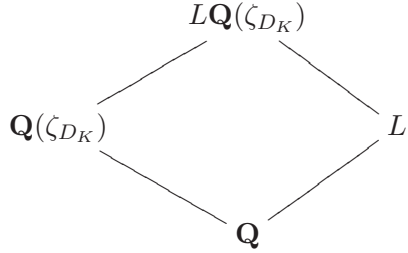
\xymatrix{& L \bfQ(\zeta_{D_{K}}) \ar@{-}[dl]\ar@{-}[dr] \\ \bfQ(\zeta_{D_{K}})\ar@{-}[dr] && L\ar@{-}[dl]\\ & \bfQ}$$ By the Tchebotarev Density Theorem we know that  the conjugates of $\Frob_p$ (as $p$ runs over all primes  not dividing $\ell D_{K}$) generate $\Gal(L\bfQ(\zeta_{D_{K}})/\bfQ)$. Furthermore for any $\tau \in \Gal(L\bfQ(\zeta_{D_K})/\bfQ)$ the image of $\tau \Frob_p \tau^{-1}$ in $\Gal(\bfQ(\zeta_{D_{K}})/\bfQ)\cong(\bfZ/D_{K}\bfZ)^{\times}$ equals $n$ if and only if $p \equiv n \pmod{D_{K}}$. Write $\varphi: \Gal(L \bfQ(\zeta_{D_{K}})/\bfQ) \twoheadrightarrow \Gal(L/\bfQ)$ for the canonical quotient map. Since $L\cap\bfQ(\zeta_{D_{K}})=K$ by Lemma \ref{intersection}, 
the image $\varphi(\Gal(L\bfQ(\zeta_{D_{K}})/\bfQ(\zeta_{D_{K}})))$ is an index two subgroup $G$ of $\Gal(L/\bfQ)$ and in fact this subgroup must be $\Gal(L/K)$. Since $\Gal(L\bfQ(\zeta_{D_{K}})/\bfQ(\zeta_{D_{K}}))$ is generated by the set $$\{\textup{conjugates of }\Frob_p \mid p\equiv 1 \pmod{D_{K}}\},$$ we get the same for its image in $\Gal(L/\bfQ)$.

Now note that $c$ itself is a conjugate of $\Frob_p$ with $p \equiv -1 \pmod{D_{K}}$.
 Consider $c \Gal(L\bfQ(\zeta_{D_{K}})/\bfQ(\zeta_{D_{K}})) \subset \Gal(L\bfQ(\zeta_{D_{K}})/\bfQ)$. On the one hand we have that $\varphi(c \Gal(L\bfQ(\zeta_{D_{K}})/\bfQ(\zeta_{D_{K}}))) =c G$ and on the other hand we must have that the image of $c \Gal(L\bfQ(\zeta_{D_{K}})/\bfQ(\zeta_{D_{K}}))$ in $\Gal(\bfQ(\zeta_{D_{K}})/\bfQ)$ is $\{-1\}$. Thus every element of $c \Gal(L\bfQ(\zeta_{D_{K}})/\bfQ(\zeta_{D_{K}}))$ (and hence also of $c G$) is a conjugate of $\Frob_p$ with $p \equiv -1 \pmod{D_{K}}$. Finally, by the Tchebotarev Density Theorem we can omit $\Frob_2$ from this set. \end{proof}

We will now finish the proof of Theorem \ref{not cong to zero}.
Let $\sigma \in G=\Gal(L/K)$. Then $c \sigma =\tau \Frob_p \tau^{-1}$ for some $\tau \in \Gal(L/\bfQ)$ and some $p \equiv -1$ mod $D_K$. 
Hence $$\tr \ov{\rho}_{\phi}(c \sigma) = \tr\ov{\rho}_{\phi}(\tau \Frob_p \tau^{-1}) = \tr \ov{\rho}_{\phi}(\Frob_p) = 0.$$ In some basis we have $\ov{\rho}_{\phi}(c)=\bmat 1\\ &-1\emat$ and thus if we write $\ov{\rho}_{\phi}(\sigma)=\bmat a&b \\ c&d \emat$, then we have $\ov{\rho}_{\phi}(c\sigma) = \bmat a&b \\ -c&-d \emat.$ Since $\tr \ov{\rho}_{\phi}(c\sigma)=0$ we get $a=d$. So, we have now proved that with respect to some fixed basis all elements $\sigma\in G$ have the property that $\ov{\rho}_{\phi}(\sigma)$ has its upper-left and its lower-right entries equal. This ensures that $\ov{\rho}_{\phi}(G)$ is abelian (see the last few lines of the proof of \cite[Proposition 8.13]{KlosinAnnInstFourier2009} for details).
 This finishes the proof of the theorem. \end{proof}

\begin{cor} \label{concretematrix} Fix an admissible base $\mB$.  If $n \not\equiv 2 \pmod{4}$, then for every $b \in \mB$  there exists $h \in S^+_{n}(\bfQ)$ with $\gamma(h)=1$ such that $\val_{\varpi}(\det h)=
\val_{\varpi}(e(-i \tr h) c_{I_{\phi}}(h, p_b))=0$. If $n \equiv 2 \pmod{4}$
then for every $b \in \mB$ there exists a prime $p \nmid 2\ell D_K$, inert in $K$ and $h\in S_n(\bfQ)$ with $\gamma(h)=-p$ such that $\val_{\varpi}(\det h)=\val_{\varpi}(e(-i \tr h) c_{I_{\phi}}(h, p_b))=0$.
\end{cor}

\begin{proof} If $n \not\equiv 2 \pmod{4}$, then \cite[Lemmas 11.1 and 11.2]{Ikeda08}  allow us to find $h$ with $\gamma(h)=1$ hence the claim follows from the definition of $\gamma$, formula (\ref{Fc formula}) and the assumption that $\ell \nmid D_K$. If $n \equiv 2 \pmod{4}$, then the existence of $h$ follows from  the proof of Theorem \ref{not cong to zero}.
\end{proof}

\section{Congruence to Ikeda lift} \label{Congruence to Ikeda lift}
We keep the notation and assumptions from section \ref{Arithm}.
 Set $\mJ(K) =\frac{1}{2}\# \OK^{\times}$. Note that $\mJ(K)=1$ when $D_K >12$. It was shown in \cite[Proposition 3.13]{Klosin15} that $\mM_{n,k}(\mK) \cong \mM_{n,k, \nu}(\mK)$ provided that $\mJ(K) \mid \nu$ and $(2n,h_K)=1$. The isomorphism between the two spaces is Hecke-equivariant and is given by a function $\Psi_{\beta}: f \mapsto \beta\otimes f$, where $\beta$ is an everywhere unramified Hecke character of $K$ of infinity type $\left(\frac{z}{|z|}\right)^{-2\nu}$. For details we refer the reader to \cite[p. 811-812]{Klosin15}. 

From now on we assume that $(h_K, 2n)=1$ and $\mJ(K) \mid \nu$ so that $\Psi_{\beta}^{-1}(I_{\phi})\in \mS_{n,2k+2m}(G_n(\hat{\bfZ}))$, and fix $\beta$ as above with $-2\nu = 2k+2m$.
We also fix a rational prime $\ell >2k+2m$ such that $\ell \nmid 2h_KD_K$ and define $\Omega_{\phi}^{\pm}\in\bfC^{\times}$ to be the integral periods associated with $\phi$ cf. e.g \cite{VatsalDuke99}. Let us note that
\begin{equation}\label{eqn:eta}
\eta_{\phi}:=\frac{\langle \phi, \phi \rangle_{\Gamma_{\phi}}}{\Omega_{\phi}^+ \Omega_{\phi}^-}\in \ov{\bfQ}^{\times}_{\ell}.
 \end{equation}
 In fact the ratio is in $\ov{\bfZ}_{\ell}$ when $\phi$ is ordinary at $\ell$ cf. e.g. \cite[Theorem 6.28]{Hida16}). Here $N_{\phi}$ denotes the level of $\phi$.
The aim of this section is to prove the following result.
Set
 $$\mV= \begin{cases} \prod_{i=2}^n\frac{L(i+2k-1, \Sym^2 \phi \otimes \chi_K^{i+1})}{\pi^{2k+2i-1}\Omega_{\phi}^+\Omega_{\phi}^-} & n=2m+1\\ \prod_{i=2}^n \frac{L(i+2k, \Sym^2 \phi \otimes \chi_K^{i})}{\pi^{2k+2i}\Omega_{\phi}^+\Omega_{\phi}^-} & n=2m. \end{cases}$$


By a result of Sturm \cite[p. 220-221]{Sturm}
for $n=2m+1$ (i.e., the weight of $\phi$ equals $2k$) we have
\be \label{St1}
\frac{L(i+2k-1, \Sym^2 \phi \otimes \chi_K^{i+1})}{\pi^{2k+2i-1}\left<\phi,\phi\right>} \in \ov{\bfQ}
\ee
and for $n=2m$ (i.e., the weight of $\phi$ equals $2k+1$) we have
\be \label{St2}
\frac{L(i+2k, \Sym^2 \phi \otimes \chi_K^{i})}{\pi^{2k+2i}\left<\phi,\phi\right>} \in \ov{\bfQ}
\ee
for $2 \leq i \leq 2k-1$ (cf. Remark \ref{power of pi} for the discrepancy in the exponent of $\pi$ between here and in \cite{Sturm}). Indeed, let us note that regardless of the parity of $n$ our points of evaluation are always in the second subset of what is called $S_1$ on page 220 of \cite{Sturm}. The inequalities there translate to exactly the above range for the values of $i$. Hence to apply this result to the $L$-factors appearing in $\mV$ we need to (and will from now on) make the assumption that $n \leq 2k-1$.
Thus in particular $\mV \in \ov{\bfQ}_{\ell}$.


\begin{thm}  \label{Ikedacong}
Assume $n\leq 2k-1$. Let $\ell \nmid 2h_KD_Ki(\phi)$ (for definition of $i(\phi)$ see section \ref{Arithm}) be a rational prime with $\ell > 2k+2m$
and write $\varpi$ for a choice of a uniformizer in some sufficiently large finite extension of $\bfQ_{\ell}$.
Let $\xi$ be a Hecke character of $K$ such that $\val_{\varpi}(\cond \xi) = 0$, $\xi_{\infty}(z) = \left(\frac{z}{|z|}\right)^{-t}$ for some $t \in \bfZ$ with $-2k-2m \leq t < \min\{-6, -4n\}$.
Then \be \begin{split}
\mU:=& \frac{\prod_{i=1}^{n}\pi^{-2n-2k-2m-t+2i-2} L^{D_K} (n+t/2+k+m-i+1, \BC(\phi)\otimes \xi^{-1}\beta)}{(\Omega_{\phi}^+\Omega_{\phi}^-)^n}\\
&\times \prod_{i=2}^{n} \frac{L(i, \chi_K^i)}{\pi^i} \times L_{D_{K}}(2n+t/2, \Psi_{\beta}^{-1}(I_{\phi}), \xi^{-1} ;\st)\end{split}
\ee belongs to $\ov{\bfQ}_{\ell}$.

 Let $\tau \in S_n(\bfQ)$ be as in Corollary \ref{concretematrix} and set $N =T D_{K} h_{K} \Nm_{K/\bfQ}(\cond \xi)$, where $T \in \bfZ$ is a generator of the inverse of the fractional ideal $\{g^* \tau^{-1} g \mid g \in \OK^n\}$ of $\bfQ$.
Assume that $\val_{\varpi}\left(T\#(\mO_{K}/N\mO_{K})^{\times}\right) =0$.
 If $\val_{\varpi}( \mU)=0$ and $b:=\val_{\varpi}(\eta_{\phi}\mV)>0$, then there exists a non-zero $f' \in \mS_{n,2k+2m}(G_n(\hat{\bfZ}))$, orthogonal to $\Psi_{\beta}^{-1}(I_{\phi})$, such that $f'\equiv \Psi_{\beta}^{-1}(I_{\phi}) \pmod{\varpi^b}$.
 \end{thm}
\begin{rem} The  assumptions in Theorem \ref{Ikedacong} ensure that the assumptions of Theorem \ref{thmmain} are satisfied. Note that in the current setup $F=\bfQ$ and since  we will apply Theorem \ref{thmmain} for the eigenform $\Psi_{\beta}^{-1}(I_{\phi})$, the weight $k$ in Theorem \ref{thmmain} is replaced by $2k+2m$ here. Also note that Lemma \ref{integers1} guarantees that the Fourier coefficients of $I_{\phi}$ are $\Oo$-integral. Furthermore, we pick $r$ in Theorem \ref{thmmain} to be $I_n$ and we choose $\tau$ to be as in Corollary \ref{concretematrix}. The conditions that $\gamma(\tau)=1$ or $p$ easily imply that condition (\ref{reldiff}) is satisfied, i.e., that $\{g^* \tau g \mid g \in \OK^n\}=\bfZ$.
Hence Corollary \ref{concretematrix} guarantees that for this choice of $r$ and $\tau$ we get $\val_{\varpi}(e(-i \tr \tau)c_{I_{\phi}}(\tau,r))=0$. Also note that $\tau$ and $T$ enter the statement of Theorem \ref{Ikedacong} only via the assumption on the valuations of $\#(\OK/N\OK)^{\times}$. 
\end{rem}

\begin{rem}
Since Theorem \ref{Ikedacong} holds for any character $\xi$ with the specified properties, it is likely that the assumption $\val_{\varpi}(\mU)=0$ always holds for some choice of $\xi$. For a more detailed explanation cf. e.g. \cite[Section 5]{BrownKeatonPJM}.  See also Section \ref{sec:examples} for examples.
\end{rem}

\begin{proof} In this proof set $\mK=G_n(\hat{\bfZ})$. As noted in Theorem \ref{thmmain} (and using (\ref{normalizedinner})) we have
\begin{equation*}
L^{\rm alg}(2n+t/2, \Psi_{\beta}^{-1}(I_{\phi}), \xi;\st) = \frac{L(2n+t/2, \Psi_{\beta}^{-1}(I_{\phi}), \xi; \st)\vol(\mF_{\mK})}{\pi^{n(2n+(2k+2m)+t+1)} \langle \Psi_{\beta}^{-1}(I_{\phi}), \Psi_{\beta}^{-1}(I_{\phi}) \rangle_{\mK}} \in \ov{\bfQ}.
\end{equation*} 
One easily checks (cf. e.g. \cite[Lemma 8.6]{Klosin15}) that $ \langle \Psi_{\beta}^{-1}(I_{\phi}), \Psi_{\beta}^{-1}(I_{\phi}) \rangle_{\mK} = \langle I_{\phi} , I_{\phi} \rangle_{\mK}.$  Using (\ref{num1}) and (\ref{den1}) we see that
$$\ov{L_{D_{K}}(2n+t/2, \Psi_{\beta}^{-1}(I_{\phi}),\xi;\st)L^{\alg}(2n+t/2, \Psi_{\beta}^{-1}(I_{\phi}),\xi;\st)}$$
equals (cf. also \cite[p. 849]{Klosin15})
\be \label{ratio1}
\frac{\pi^{-n(2n+(2k+2m)+t+1)}\prod_{i=1}^{n} L^{D_K} (n+t/2+k+m-i+1, \BC(\phi)\otimes \xi^{-1}\beta)}{\left<\phi, \phi\right> \prod_{i=2}^n  L(i+2k-1, \Sym^2 \phi \otimes \chi_K^{i+1})L(i, \chi_K^i)\Gamma_{\bfC}(i+2k-1)\Gamma_{\bfC}(i)^2}\vol(\mF_{\mK})
\ee
for $n=2m+1$ and
\be \label{ratio2}
\frac{\pi^{-n(2n+(2k+2m)+t+1)}\prod_{i=1}^{n} L^{D_K} (n+t/2+k+m-i+1, \BC(\phi)\otimes\xi^{-1}\beta)}{\left<\phi, \phi\right> \prod_{i=2}^n  L(i+2k, \Sym^2 \phi \otimes \chi_K^{i})L(i, \chi_K^i)\Gamma_{\bfC}(i+2k)\Gamma_{\bfC}(i)^2}\vol(\mF_{\mK})
\ee
for $n=2m$.

Using the definition of $\Gamma_{\bfC}$ expressions (\ref{ratio1}) and (\ref{ratio2}) become (here $u$ is some $\varpi$-adic unit)
 \be \label{ratio1'}
 \frac{\pi^{-n(n+(2k+2m)+t+1)}\prod_{i=1}^{n} L^{D_K} (n+t/2+k+m-i+1, \BC(\phi)\otimes \xi^{-1}\beta)u }{\left<\phi, \phi\right> \prod_{i=2}^n L(i, \chi_K^i)\pi^{-i} L(i+2k-1, \Sym^2 \phi \otimes \chi_K^{i+1})\pi^{-2k-2i+1}}\frac{\vol(\mF_{\mK})}{\pi^{n^2}}
 \ee
 for $n=2m+1$ and
\be \label{ratio2'}
\frac{\pi^{-n(n+(2k+2m)+t+1)}\prod_{i=1}^{n} L^{D_K} (n+t/2+k+m-i+1, \BC(\phi)\otimes \xi^{-1}\beta)u }{\left<\phi, \phi\right> \prod_{i=2}^n  L(i, \chi_K^i)\pi^{-i} L(i+2k, \Sym^2 \phi \otimes \chi_K^{i})\pi^{-2k-2i}}\frac{\vol(\mF_{\mK})}{\pi^{n^2}}
\ee
for $n=2m$.

For a Hecke character $\psi$ of $K$ of infinity type $(z/|z|)^u$ ($u\leq 0$) we will write $$g_{\psi}=\sum_{j=1}^{\infty} a_{g_{\psi}}(j)q^j\quad  \textup{with}\quad   a_{g_{\psi}}(j)=\sum_{\substack{\fa\subset \OK \hspace{1pt} \textup{ideal}\\ N(\fa)=j}}j^{u/2}\psi(\fa)$$ for the associated modular form of weight $-u+1$ (which is a cusp form if $u<0$). Observe that we have $L(s,g_{\psi}) = L(s - u/2, \psi)$.  It is easy to check one has
 $$L^{D_K}(s, \BC(\phi)\otimes \psi) = L^{D_K}(s-u/2, \phi \otimes g_{\psi}),$$
 where $L^{D_{K}}(s, \phi \otimes g_{\psi})$ is the convolution $L$-function which for $s \in \bfC$ with sufficiently large real part is defined by $$\prod_{p \nmid D_K}\{ (1-\alpha_{\phi}\alpha_{g_{\psi}} p^{-s}) (1-\alpha_{\phi}\beta_{g_{\psi}} p^{-s}) (1-\beta_{\phi}\alpha_{g_{\psi}}p^{-s}) (1-\beta_{\phi}\beta_{g_{\psi}} p^{-s})\}^{-1}$$ with $\alpha_{\phi}, \beta_{\phi}$ and $\alpha_{g_{\psi}}, \beta_{g_{\psi}}$ the Satake parameters of $\phi$ and $g_{\psi}$ respectively where, as above, the Satake parameters are normalized arithmetically.

We note that the character $\xi^{-1}\beta$ has infinity type $(z/|z|)^{2k+2m+t}$, hence the character $\xi \beta^{-1}$ has infinity type $ (z/|z|)^{-2k-2m-t}$ and the number $-2k-2m-t$ is by our assumption on $t$  a  negative number. Thus the cusp form corresponding to $\xi\beta^{-1}$ is $g=g_{\xi\beta^{-1}}$ 
 which is of weight $2k+2m+t+1>0$.
 Hence the $L$-function in the numerator of (\ref{ratio1}) equals

\be \begin{split}
&\ov{L^{D_K}(n+t/2+k+m-i+1, \BC(\phi)\otimes \xi\beta^{-1})}\\
=& \ov{L^{D_K}(n+t+2k+2m-i+1, \phi\otimes g)}\\
=&L^{D_K}(n+t+2k+2m-i+1, \phi^c\otimes g^c),
\end{split}\ee where $c$ denotes conjugating the Fourier coefficients.

 Since $t<\textup{min}\{-6,-4n\}$ we see that for every $i\in \{1,2,\dots, n\}$ the point of evaluation $n+2k+2m+t-i+1$ satisfies
 $$2k+2m+t < 2k+2m+t+n -i+1<2k,$$
 i.e., the points of evaluation lie strictly between the weights of $\phi$ and $g$, hence they are critical points in the sense of Deligne. We have the following result due to Shimura.

\begin{thm}
 If $\phi$ and $\phi'$ are two cuspidal eigenforms of weights $l,l'$ respectively, and $l>l'$ then for all integers $M$ such that $l' \leq M <l$ one has $$\frac{\pi^{l'-1-2M}L(M,\phi\otimes \phi')}{\left<\phi,\phi\right>} \in \ov{\bfQ}.$$
 The values of $M$ in the above ranges are critical. \end{thm}
\begin{proof} This is stated on \cite[p. 218]{ShimuraMathAnn77}. For the discrepancy in the exponent of $\pi$ between our statement and the formula in [loc.cit.], see Remark \ref{power of pi}. \end{proof}

 Applied to our case ($l=2k$ if $n = 2m+1$, $l = 2k+1$ if $n = 2m$, and $l'=2k+2m+t+1$) this implies that
 $$\frac{\pi^{-2n-2k-2m-t+2i-2}L(n+2k+2m+t-i+1, \phi\otimes g)}{\left<\phi,\phi\right>} \in \ov{\bfQ} \quad \textup{for $i \in \{1,2,\dots, n\}$}.$$
 Note that the power of $\pi$ in $\mU$ (not involved in normalizing the Dirichlet $L$-function) equals \begin{align*}
 \sum_{i=1}^n (-2n-2k-2m-t+2i-2) &= -n(2n+2k+2m+t+2) +2\sum_{i=1}^n i \\
    &= -n(n+2k+2m+t+1)
 \end{align*}
 and  that $L(i, \chi_K^i)\pi^{-i} \in \ov{\bfQ}$.  Thus we have proved that $\mU \in \ov{\bfQ}_{\ell}$.

Hence both (\ref{ratio1'}) and (\ref{ratio2'}) equal (taking into account \eqref{eqn:eta})
$$u\frac{i(\phi)}{\eta_{\phi}}\frac{\mU}{\mV}\frac{\vol(\mF_{\mK})}{\pi^{n^2}},$$ where $u$ is a $\varpi$-adic unit.
To summarize we have shown
\begin{align*}
\val_{\varpi}\left(L^{\rm alg}(2n+t/2, \Psi_{\beta}^{-1}(I_{\phi}), \xi^{-1};\st)\frac{\pi^{n^2}}{\vol(\mF_{\mK})}\right) &= \val_{\varpi}(
\mU/(\eta_{\phi}\mV))\\
    &= -\val_{\varpi}(\eta_{\phi}\mV)\\
    & =-b.
\end{align*}
The result now follows directly from Theorem \ref{thmmain} and the fact that $f'\neq 0$ follows from Remark \ref{couldbezero} and Theorem \ref{not cong to zero}. \end{proof}

\begin{rem} Theorem \ref{not cong to zero} ensures that there is a genuine congruence to depth $b$ between $I_{\phi}$ and some orthogonal form $f'$, i.e., that no part of that congruence comes from the fact that $I_{\phi}$ itself is congruent to zero to some depth. We also note that Theorem \ref{not cong to zero} is indeed necessary for the construction of such a congruence, because our method does not allow for any rescaling of the lift as such a rescaling would change the inner products used in the derivation of the congruence. \end{rem}

\section{Congruence with respect to integral periods} \label{congint}

It is possible that the congruence constructed in Theorem \ref{Ikedacong} is ``inherited'' directly from a congruence between $\phi$ and another newform $\phi'\in S_{2k}(1)$ (or $S_{2k+1}(D_K, \chi_K)$). 

 It is well-known \cite{HidaSugaku89} that for $\ell$ as in the last section (i.e., in particular $\ell>2k+1$, $\ell \nmid i(\phi)$) the mod $\ell^r$ congruences between $\phi$ and another $\phi'$ are controlled  by the ratio \be \label{cong id gen} \eta_{\phi}:=\frac{\langle \phi,\phi\rangle_{\Gamma_{\phi}}}{\Omega_{\phi}^+\Omega_{\phi}^-}\in \ov{\bfZ}_{\ell}\ee at least when $\phi$ is ordinary (see e.g., \cite[Section 5]{BrownCompMath07} for definitions of $\Omega_{\phi}^{\pm}$ as well as the discussion of the ratio).

 In this section we will derive  (under certain assumptions)  a stronger version of Theorem \ref{Ikedacong} where the condition $\val_{\ell}(\eta_{\phi}\mV)>0$ is replaced by $\val_{\ell}(\mV)>0$, and in return the resulting congruence is between $I_{\phi}$ and an automorphic form $f' \in \mM_{n,2k+2m}(G_n(\hat{\bfZ}))$, which does not arise as an Ikeda lift.
 However, in doing so, we will make use of the following conjecture.

\begin{conj}\label{Skinner1} Fix a rational prime $\ell$.  Let $f \in \mS_{n, k, \nu}(\mK)$ be a Hecke eigenform of central character $\omega$ for $\mK \subset G_{n}(\bfA_{\bfQ,\bff})$ an open compact subgroup. Then there exists a continuous semi-simple representation $\rho_f: G_K \to \GL_{2n}(\ov{\bfQ}_{\ell})$ such that \begin{itemize}
\item[(i)] $\rho_f$ is unramified at all finite places not dividing $D_K$;
\item[(ii)] One has $L^{D_K}(s, \rho_f) = L^{D_K}\left(s, \BC(f)\otimes \omega^c; {\rm st}\right)$,
\item[(iii)] If $\ell \nmid D_K$, then for any $\fp \mid \ell$, the representation $\rho_f|_{D_{\fp}}$ is crystalline.
\end{itemize}
\end{conj}

\begin{rem} The above conjecture is widely regarded as a known result, however the only reference we know of in which the existence of Galois representations attached to automorphic forms on $G_n(\bfA_{\bfQ})$ is proven is the article of Skinner
 \cite[Theorem B]{SkinnerUnitaryGaloisReps}.
 However, the proof excludes the parallel weight case, and while the author states that this hypothesis may be relaxed, he goes on to say that this case would not be addressed in [loc.cit.]. Let us also mention that a similar remark to Remark \ref{discrepancy} regarding different normalization of $L$-functions concerns \cite{SkinnerUrbanInventMath14}. \end{rem}



From now on as in the last section, let $\ell \nmid 2D_K$, $\ell >n+2k-1$ be a prime and $\phi$ be a newform in $S_{2k}(1)$ (if $n=2m+1$) or in $S_{2k+1}(D_K, \chi_K)$ (if $n=2m$). Let $\rho_{\phi}: G_{\bfQ} \to \GL_2(\ov{\bfQ}_{\ell})$ be the $\ell$-adic Galois representation associated to $\phi$ by Deligne et al. and write  $\rho_{\phi}(j)$ for its $j$th Tate twist. We will also assume that the residual Galois representation $\ov{\rho}_{\phi}: G_{\bfQ} \to \GL_2(\ov{\bfF}_{\ell})$ is irreducible when restricted to $G_K$.
This combined with (\ref{num1}) and Conjecture \ref{Skinner1} implies that $\rho_{I_{\phi}}$ is isomorphic to
\be \label{eqrepr}
\epsilon^{-k-m+n}\otimes \bigoplus_{j=0}^{n-1} \rho_{\phi}(j)|_{G_K}
\ee
where $\epsilon$ is the $\ell$-adic cyclotomic character.


Let $S_0'=\{\phi_1=\phi, \phi_2, \dots, \phi_{r'}\}$ be the subset of a  basis of newforms of $S_{2k}(1)$ (if $n=2m+1$) or of $S_{2k+1}(D_K, \chi_K)$ (if $n=2m$) consisting of forms congruent to $\phi$ (mod $\varpi$). Write $S_1':=\{f_1=I_{\phi}, \dots, f_{r'}=I_{\phi_{r'}}\}$ for the corresponding set of Ikeda lifts.
Since $\ov{\rho}_{\phi}|_{G_K}$ is irreducible, it must have a non-abelian image, hence it follows from Theorem \ref{not cong to zero} that all $f_i\in S'_1$ are non-zero.
 We choose a subset $S_0$ of $S'_0$ so that the elements of $S_1:=\{f_i =I_{\phi_i}\in S'_1\mid \phi_i \in S_0\}$ are  linearly independent and span the same subspace of $\mS_{2k+2m}(G_n(\hat{\bfZ}))$ as $S'_1$. Renumbering the elements of $S'_0$ and hence of $S'_1$ if necessary we may assume that $S_0=\{\phi_1=\phi, \phi_2, \dots, \phi_r\}$ and $S_1=\{f_1=I_{\phi}, f_2, \dots f_r\}$ for some $r \leq r'$.
By Corollary \ref{cong between phis} we know that $f_i \equiv f_1$ (mod $\varpi$) for all $1\leq i \leq r$.  We complete this set to a basis  $f_1, \dots, f_r, f_{r+1}, \dots, f_s$ of the subspace $\mW$ of $\mS_{2k+2m}(G_n(\hat{\bfZ}))$ containing all eigenforms  which are congruent to $f_1$ (mod $\varpi$)  (in the sense of Definition \ref{def of cong})
and  note that by Theorem \ref{not cong to zero} we must have $f_i \not\equiv 0 \pmod{\varpi}$ for all $1\leq i \leq s$.


For $\sigma \in G_K$ let $\sum_{j=0}^{2n} c_j(i,\sigma)X^j\in \Oo[X]$ be the characteristic polynomial  of $\rho_{f_i}(\sigma)$, where $\rho_{f_i}$ is the Galois representation attached to $f_i$ (Conjecture \ref{Skinner1}). Here $\Oo$ is the valuation ring of some sufficiently large finite extension of $\bfQ_{\ell}$. Put $c_j(\sigma):= \bmat c_j(1,\sigma) \\ \vdots \\ c_j(s, \sigma)\emat \in \Oo^s$ for $j=0,1, \dots, 2n.$ Let $\bfT$ be the $\Oo$-subalgebra of $\Oo^s$ generated by the set $\{c_j(\sigma) \mid 0 \leq j \leq 2n, \sigma \in G_K\}$.  The algebra $\bfT$ acts on $\mW$ by 
 $$Tf=\bmat t_1\\ \vdots \\ t_s\emat \sum_{i=1}^s \alpha_i f_i:= \sum_{i=1}^s \alpha_i t_i f_i.$$
We will say that $\bfT$ \emph{preserves the integrality of Fourier coefficients} if whenever $f \in \mW$ has $\Oo$-integral Fourier coefficients, so does $Tf$ for all $T \in \bfT$ (cf. Definition \ref{def of cong}).

\begin{rem} It follows from the Tchebotarev Density Theorem that $\bfT$ is generated by the set $\{c_j(\Frob_{\fp}) \mid 0 \leq j \leq 2n, \fp \nmid D_K\ell\}$. By Conjecture \ref{Skinner1} (ii), each $c_j(\Frob_{\fp})$ is a polynomial in Hecke operators. The action of each such Hecke operator on Fourier coefficients (albeit tedious) should be straightforward to compute (we refer the reader for example to \cite[(5.6)]{KlosinAnnInstFourier2009} where this is done for $n=2$) and from this one should be able to  see that these polynomials preserve the integrality of Fourier coefficients. Hence  in fact we expect that $\bfT$ always preserves the integrality of Fourier coefficients, but we do not pursue the full proof here since the theorem   below (for which the integrality is used) is already conditional on Conjecture \ref{Skinner1} and a full proof may be computationally involved.   \end{rem}


\begin{thm}\label{integral1} Assume that Conjecture \ref{Skinner1} holds. Keep the assumptions of Theorem \ref{Ikedacong}, but replace the definition of $b = \val_{\varpi}(\eta_{\phi} \mV)$ with $b = \val_{\varpi}(\mV)$. Suppose furthermore that $\phi$ is ordinary at $\ell$ with  $\ov{\rho}_{\phi}|_{G_K}$  absolutely irreducible and that the algebra $\bfT$ preserves the integrality of Fourier coefficients.
Then the claims of Theorem \ref{Ikedacong} hold true, but the form $f' \in \mM_{n,2k+2m}(G_n(\hat{\bfZ}))$ congruent to $\Psi_{\beta}^{-1}(I_{\phi})\pmod{\varpi^b}$ is not an Ikeda lift of any newform $\phi'\in S_{2k}(1)$ if $n=2m+1$ (resp. $\phi' \in S_{2k+1}(D_K, \chi_K)$ if $n=2m$).
\end{thm}

\begin{proof} As in \cite{BrownCompMath07} to achieve our goal we will show the existence of a certain `idempotent-like' Hecke operator in the Hecke algebra on $G_n(\AQ)$. However, instead of deducing its existence from explicit relations between Hecke operators on $G_n$ and $\GL_2$ (as in [loc.cit.]) we will construct it from the knowledge of the Galois representation of $I_{\phi}$ (a variant of this method was used in
\cite[Section 5.5]{KlosinAnnInstFourier2009}).
As the arguments for $n=2m$ are entirely analogous, we limit ourselves to the case $n=2m+1$.

Denote by $\bfT_f$ the quotient of $\bfT$ defined in the same way as $\bfT$ but with $r$ in place of $s$.
After choosing some $G_{\bfQ}$-invariant lattice we may assume that $\rho_{\phi_i}$ is valued in $\GL_2(\Oo)$. 
Write $\rho: \Oo[G_K] \to \Mat_{2}(\Oo^r)$ for the $\Oo$-algebra map sending $g=\sum_l a_l g_l \in \Oo[G_K]$ to $\sum_l a_l \bmat \rho_{\phi_1}(g_l)\\ \vdots \\ \rho_{\phi_r}(g_l)\emat$. By a standard argument (using the irreducibility of $\ov{\rho}_{\phi_i}|_{G_K}$) one can show that the image of $\rho$ is contained in $\Mat_2(\bfT_0)$, where $\bfT_0$ is the $\Oo$-Hecke algebra acting on $S_{2k}(1)$ (cf. e.g. \cite[Lemma 3.27]{DarmonDiamondTaylor95}). In particular $\tr \rho(g)$ can be interpreted as a Hecke operator. It thus follows from \cite[Proposition 8.14]{KlosinAnnInstFourier2009} that there exists $g_0 \in \Oo[G_K]$ such that $\tr\rho(g_0)\phi_1=\frac{\langle \phi,\phi\rangle}{\Omega_{\phi}^+\Omega_{\phi}^-}\phi_1$ and $\tr \rho(g_0)\phi_i=0$ for all $i>1$. The proposition in [loc.cit.] as stated applies only to $K=\bfQ(i)$, but the deformation-theoretic arguments in the proof are not sensitive to these restrictions and the result of Hida they reduce the proof to is also valid in full generality.

Now let $\tilde{\rho}: \Oo[G_K] \to \Mat_{2n}(\Oo^r)$ be  the $\Oo$-algebra map sending $g=\sum_l a_l g_l \in \Oo[G_K]$ to $\sum_l a_l \bmat \rho_{I_1}(g_l)\\ \vdots \\ \rho_{I_r}(g_l)\emat.$ By (\ref{eqrepr}) we see that $\tilde{\rho}$ decomposes as a direct product of Tate twists of $\rho$, hence one concludes that the image of $\tilde{\rho}$ is contained  in $\Mat_{2n}(\bfT_f)$. Using (\ref{eqrepr}) again and the fact that $\phi$ is ordinary at $\ell$ and working as in the proof of \cite[Proposition 5.14]{KlosinAnnInstFourier2009} we see that there exists a projector $e \in \Oo[G_K]$ such that $$\tilde{\rho}(e) = \bmat \rho(e) \\ & \rho(1)(e) \\ && \ddots \\ &&& \rho(n-1)(e) \emat= \bmat \rho(e) \\ & 0 \\ && \ddots \\ &&& 0 \emat.$$ Set $T_{\phi}:= \tr \tilde{\rho}(eg_0)$ with $g_0$ as above and let $T$ be any lift of $T_{\phi}$ to $\bfT$. It is clear that this $T$ has the desired properties. Now applying $T$ to both sides of (\ref{fir}) and amending the periods in the $L$-functions appearing in $C_{f_1}$ we get the desired result. We refer the reader to \cite{BrownCompMath07} or \cite{KlosinAnnInstFourier2009} for details.
\end{proof}

\section{Speculations about consequences for the Bloch-Kato conjecture} \label{Speculations about consequences for the Bloch-Kato conjecture}

Let us now specialize to the case $n=3$. 
The relevant $L$-values in this case are $$L_1:=\frac{ L(2k+1, \Sym^2 \phi \otimes \chi_K)}{\pi^{2k+3}\Omega_{\phi}^+\Omega_{\phi}^-}\in \ov{\bfQ}, \quad L_2:=\frac{ L(2k+2, \Sym^2 \phi)}{\pi^{2k+5}\Omega_{\phi}^+\Omega_{\phi}^-}\in \ov{\bfQ}.$$
For $i=1,2$ set $v_i:= \val_{\ell}(\#\Oo/L_i)$. 
The ($\varpi$-part of the) Bloch-Kato conjecture for a $G_F$-module $V$ with divisible coefficients(with $F$ a number field) predicts that $\val_{\ell}(\#H^1_f(F, V^{\vee}(1)))$ and  $\val_{\ell}(\#\Oo/L^{\rm alg}(V,0))$ (where $L^{\rm alg}(V,0)$ is the appropriately normalized $L$-value of $V$ at zero) should be related (in fact equal up to certain canonically defined factors - see e.g. \cite[Section 9]{Klosin15} for a precise statement). We have (to ease notation here and below we assume that the modules have divisible coefficients)
\be
\begin{split} V_1:= \ad^0\rho_{\phi} (2)\otimes \chi_K, & \quad V_1^{\vee} = \ad^0\rho_{\phi}(-2) \otimes \chi_K,\\
V_2:= \ad^0\rho_{\phi}(3), & \quad V_2^{\vee}=\ad^0\rho_{\phi}(-3).
\end{split}
\ee
So, the Bloch-Kato conjecture for $V_1$ relates $\val_{\ell}(\# H^1_f(\bfQ, \ad^0\rho_{\phi}(-1)\otimes \chi_K))$ to
$\val_{\ell} (\#\Oo/L^{\rm alg}(\ad^0\rho_{\phi}(2)\otimes\chi_K,0)) = v_1$, 
and for $V_2$ it relates
$ \val_{\ell}(\#H^1_f(\bfQ, \ad^0\rho_{\phi}(-2)))$ to $\val_{\ell} (\#\Oo/L^{\rm alg}(\ad^0\rho_{\phi}(3), 0))=v_2$. 

 By Theorem \ref{integral1} there exists $f'\in \mS_{2k+2}(G_3(\hat{\bfZ}))$ orthogonal to the space spanned by all the Ikeda lifts such that $f'\equiv I_{\phi}$ (mod $\varpi^{v_1+v_2}$).
Assume for simplicity that $v_1=0$ and $v_2=1$. Assuming the validity of Conjecture \ref{Skinner1} we get an $\ell$-adic Galois representation attached to $f'$ which is unramified away from $\ell$ and crystalline at $\ell$. Let us assume it is also irreducible. Then by a standard argument of Ribet \cite{RibetInvent76} we can find a Galois-invariant lattice inside this representation such that with respect to that lattice its mod $\ell$-reduction $\rho: G_K \to \GL_6(\ov{\bfF}_{\ell})$ has the form
$$\rho(g) =\ov{\epsilon}(g)^{2-k} \bmat \ov{\rho}_{\phi}(g) & a(g) & b(g)\\ & \ov{\rho}_{\phi}(g)\ov{\epsilon}(g) & c(g) \\ & &\ov{\rho}_{\phi}(g)\ov{\epsilon}^2(g)\emat.$$
First note that both $a$ and $c$ give rise to elements in $H^1_f(K, \ad^0\rho_{\phi}(-1)|_{G_K})$. One has a natural splitting $$H^1_f(K,\ad^0\rho_{\phi}(-1)|_{G_K}) = H^1_f(K,\ad^0\rho_{\phi}(-1)|_{G_K})^+ \oplus H^1_f(K,\ad^0\rho_{\phi}(-1)|_{G_K})^-,$$ where the superscripts indicate the sign of the action of the complex conjugation. 
Suppose now one can show that $a,c \in H^1_f(K,\ad^0\rho_{\phi}(-1)|_{G_K})^-$. Let us explain how one can use our result to conclude that the validity of the Bloch-Kato conjecture for $V_1$ gives evidence for its validity for $V_2$. The inflation-restriction sequence gives an identification $$H^1(K,\ad^0\rho_{\phi}(-1)|_{G_K})^- = H^1(\bfQ,  \ad^0\rho_{\phi}(-1)|_{G_K}\otimes \chi_K)$$ and using this one can show that the classes $a,c$ lie in the Selmer group of $\ad^0\rho_{\phi}(-1)|_{G_K}\otimes \chi_K$ which by the Bloch-Kato conjecture for $V_1$ and our assumption that $v_1=0$ is trivial. So, $a,c$ must be trivial classes. Hence for $\rho$ to be non-semi-simple, $b$ has to give rise to a non-trivial class in $H^1(K, \ad^0\rho_{\phi}(-2)|_{G_K})$. Splitting $a$ and $c$ and rearranging the diagonal elements in $\rho$, we see that $\bmat \ov{\rho}_{\phi}|_{G_K} & b \\ & \ov{\rho}_{\phi}(2)|_{G_K}\emat$ is a subrepresentation of $\rho$ and hence also crystalline. Thus, $b$ in fact gives rise to a non-zero element in the Selmer group of $\ad^0\rho_{\phi}(-2)|_{G_K}$. This provides evidence for the Bloch-Kato conjecture for $V_2$. In fact this argument can be applied with $V_1$ and $V_2$ switched. Furthermore, with some more work it could be extended to arbitrary values of $v_1$ and $v_2$ to obtain the full equivalence of the validity of the  Bloch-Kato conjectures for $V_1$ and $V_2$. However, we do not proceed with the  general argument here, because we are unable to show that the classes $a$ and $c$ indeed lie in the minus part of the Selmer group, which at this stage reduces our argument to a mere speculation.

\begin{rem} In \cite{BrownKeatonPJM}, Keaton and the first-named author construct congruences for the (symplectic) Ikeda lift on the group $\Sp_{2n}$ of an elliptic modular form $\phi \in S_k(1)$ when $k$ and $n$ are even. These are controlled by the  $L$-value $L_{\rm alg}((k+n)/2, \phi) \prod_{j=1}^{n/2-1}L_{\rm alg}(2j+k, \Sym^2 \phi)$, where the subscript $\textup{alg}$ indicates the algebraic part. The paper also discusses the application of the result to the Bloch-Kato conjecture for $\phi$. Our current result can be viewed as partly complementary to that of \cite{BrownKeatonPJM} in the sense that it constructs congruences for the Ikeda lift of $\phi \in S_k(1)$ to $\U(n,n)$ for odd values of $n$, while \cite{BrownKeatonPJM} construct congruences for the Ikeda lift of $\phi$ to $\Sp_{2n}$ for even values of $n$. \end{rem}

\section{Examples}\label{sec:examples}

In this final section we provide examples of our main theorem.  We fix $K = \bfQ(\sqrt{-3})$ throughout the calculations so that $\chi_{K} = \left(\frac{-3}{\cdot}\right)$. Note that here that $h_{K} = 1$ which simplifies many of the assumptions.  We consider the case where $m=2$ and $n = 5$.  Furthermore, observe that we have $\#\mO_{K}^{\times} = 6$, so $\mJ(K) = 3$ in this case.

Fix $\phi \in S_{26}(\SL_2(\bfZ))$ to be the unique normalized eigenform, i.e.,
\begin{equation*}
\phi = q - 48q^{2} - 195804q^{3} - 33552128q^{4} - 741989850q^{5} +
9398592q^{6} +
\cdots.
\end{equation*}
Observe that this gives $\nu = -k-m = -15$, so $\mJ(K) \mid \nu$ as required.

%
Set
\begin{equation*}
L_{\alg}(j+25,\Sym^2 \phi \otimes \chi_{K}^{j+1}) = \frac{L(j+25,\Sym^2 \phi \otimes \chi_{K}^{j+1})}{\pi^{25+2j} \langle \phi, \phi \rangle}.
\end{equation*}
To calculate the values of $L_{\alg}$ listed below, we made use of SAGE \cite{sage} and Dokchitser's ComputeL program based on the work in \cite{DokchitserComputeL}. Using ComputeL, we computed the values $L_{\alg}$ and looked at continued fraction expansions of these with increasing levels of precision.  From this, it becomes clear where a ``cut-off'' in the continued fraction should occur to give the rational value of $L_{\alg}$. The same methods are used to compute the algebraic values of the convolution $L$-functions below as well. Thus, we have not proven the values below are correct, but have strong evidence they are the correct values.

 Using this method we have
\begin{align*}
L_{\alg}(27,\Sym^2 \phi \otimes \chi_{K}) &= \frac{2^{37} \cdot 523}{3^{33} \cdot 5^5 \cdot 7^2 \cdot 11 \cdot 13 \cdot 23}\\
L_{\alg}(28, \Sym^2 \phi \otimes \chi_{K}^2)
    &=\frac{2^{38}\cdot 5 \cdot 7^{2} \cdot 11 \cdot 13 \cdot 23 \cdot 31 \cdot 137}{3^{38} \cdot 5^6 \cdot 7^4 \cdot 11^2 \cdot 13^2 \cdot 23^2}\\
L_{\alg}(29, \Sym^2 \phi \otimes \chi_{K}) &= \frac{2^{38} \cdot 67 \cdot 139 \cdot 1609}{3^{40} \cdot 5^5 \cdot 7^2 \cdot 11^2 \cdot 13^2 \cdot 23^2}\\
L_{\alg}(30,\Sym^2 \phi \otimes \chi_{K}^2) &=  \frac{2^{31}\cdot 5 \cdot 11^{2} \cdot 19 \cdot 3463 \cdot 6761}{3^{44} \cdot 5^{6} \cdot 7^3 \cdot 11^2 \cdot 13^2 \cdot 19 \cdot 23^2 \cdot 29}.
\end{align*}
Observe since $\phi$ the only newform of weight 26 and full level, vacuously there are no non-trivial congruences between $\phi$ and other newforms of weight 26 and full level and thus $\eta_{\phi}$ is a unit.  This gives $\val_{\ell}(\mV) = 1$ for $\ell \in \mL:=\{31, 67, 137, 139, 523, 1609,3463,6761\}$. Moreover, $f$ is ordinary at all $\ell \in \mL$. We can ignore the condition that the residual Galois representations $\ov{\rho}_{\ell}$ be irreducible as that was only required to produce the Hecke operator that guaranteed we produce a congruence to a form that is not an Ikeda lift; as explained above $\phi$ is not congruent to any other newforms of the same weight and level so one does not have to worry about such congruences to other Ikeda lifts.

Clearly we have $\ell > 2k+2m = 30$ for all $\ell \in \mL$.  Observe that $\mL$ is relatively prime to $2 h_{K}D_{K}$ as required where we say $A$ is relatively prime to $\mL$ if $\val_{\ell}(A) = 0$ for all $\ell \in \mL$.  We also have that $\mL$ is relatively prime to $T \#(\mO_{K}/3\mO_{K})^{\times} = \#(\mO_{K}/3\mO_{K})^{\times}$. Finally, note that $\mL$ is relatively prime to $A_{3}$. Thus, it only remains to show we can choose $\xi$ so that $\mU$ is relatively prime to $\mL$.

Recall that $\xi$ is a Hecke character of $K$ so that $\xi_{\infty}(z) = \left(\frac{z}{|z|}\right)^{-t}$ for some $t \in \bfZ$ satisfying $-30 \leq t < -20$. For our purposes we choose $\xi$ to be an everywhere unramified character with $t = -24$.  Recalling that $\beta$ is an everywhere unramified Hecke character of $K$ of infinity type $\left(\frac{z}{|z|}\right)^{30}$, we have $g:= g_{\xi \beta^{-1}}\in S_{7}(3,\chi_{K})$ \cite[Theorem 12.5]{IwaniecBook}.  In fact, one can check via SAGE that $g$ is the unique newform in $S_7(3,\chi)$, i.e.,
\begin{equation*}
g = q - 27q^{3} + 64q^{4} - 286q^{7} + 729q^{9} - 1728q^{12} + 506q^{13} +
4096q^{16} - 10582q^{19} + \cdots.
\end{equation*}
Observe that $g^{c} = g$.

To deal with $\mU$, we first consider
\begin{equation*}
\frac{\prod_{j=1}^{5} \pi^{-13+2j} L(9-j, \BC(\phi) \otimes \chi_{K})}{\langle \phi, \phi\rangle^{5}}.
\end{equation*}
One thing to note here is that while the main theorem has the Euler factors at $D_{K}$ removed, so the Euler factor at 3 for us, this is not necessary in the case that $\phi$ has full level.  One can see this by running through the proof of \cite[Theorem 18.1]{Ikeda08}.
Using the relation between $L(s, \BC(\phi) \otimes \chi_{K})$ and $L(s,\phi \otimes g)$, we see we want to consider
\begin{equation*}
\prod_{j=1}^{5} \frac{L(j, \phi \otimes g)}{\pi^{6-2j} \langle \phi, \phi \rangle}.
\end{equation*}
Set
\begin{equation*}
L_{\alg}(j,\phi \otimes g) = \frac{L(j, \phi \otimes g)}{\pi^{6-2j} G(\chi_{3}) i^{33-2j} \langle \phi, \phi \rangle}.
\end{equation*}
We need to calculate $L_{\alg}(j, \phi \otimes g)$ for $j = 7,8, 9, 10, 11$ and see these are relatively prime to the elements in $\mL$.   We have
\begin{align*}
L_{\alg}(7,\phi \otimes g) &= \frac{7^{2} \cdot 13 \cdot 17 \cdot 19 \cdot 107}{2^{5} \cdot 3^{4} \cdot 5^{18}}\\
L_{\alg}(8,\phi \otimes g) &= \frac{7 \cdot 17 \cdot 127 \cdot 7607}{2^{3} \cdot 3^{6} \cdot 5^{18}}\\
L_{\alg}(9,\phi \otimes g) &= \frac{2 \cdot 109 \cdot 1428767}{3^{7} \cdot 5^{16} \cdot 7 \cdot 23}\\
L_{\alg}(10,\phi \otimes g) &= \frac{2^{7} \cdot 13 \cdot 853}{3^{10} \cdot 5^{12} \cdot 23}\\
L_{\alg}(11,\phi \otimes g) &= \frac{2^{9} \cdot 47 \cdot 2069}{3^{12} \cdot 5^{12} \cdot 7 \cdot 23}.
\end{align*}
Finally we consider the values of $L_{\alg}(j,\chi_{K}^{j}):= \frac{L(j,\chi_{K}^{j})}{\pi^{j}}$ for $j = 2, 3, 4, 5$.  The values of $L_{\alg}(j,\chi_{K}^{j})$ for $j$ odd can be obtained directly from SAGE; the others are easily obtained from well-known values of the Riemann zeta function.
\begin{align*}
L_{\alg}(2,\chi_{K}^2) &= (1-3^{-2})\zeta(2) = \frac{2^2}{3^3}\\
L_{\alg}(3, \chi_{K}^3) &= \frac{2^2 \cdot \sqrt{3}}{3^5} \\
L_{\alg}(4,\chi_{K}^4) &= (1-3^{-4})\zeta(4) = \frac{2^3}{3^6}\\
L_{\alg}(5,\chi_{K}^5) &= \frac{2^2 \cdot \sqrt{3}}{3^7}.
\end{align*}
Thus, we have $\mU$ is relatively prime to $\mL$.

Combining all of these computations we see Theorem \ref{Ikedacong} gives a congruence modulo $\ell$ for each $\ell \in \mL$ between the Ikeda lift of $\phi$ and a Hermitian modular form of weight 30 and full level that is orthogonal to $I_{\phi}$ and is not an Ikeda lift.

\bibliographystyle{plain}
 \bibliography{mybib}

\end{document}